\documentclass[11pt]{article}
\usepackage{amssymb,amsmath}
\usepackage{graphicx}

\usepackage{latexsym}

\newfont{\bb}{msbm10 at 11pt}
\newfont{\bbsmall}{msbm8 at 8pt}

\newcommand{\h}{\widehat}

\newcommand{\R}{\mbox{\bb R}}

\newcommand{\N}{\mbox{\bb N}}

\newcommand{\Z}{\mbox{\bb Z}}
\newcommand{\M}{\mbox{\bb M}}
\newcommand{\B}{\mbox{\bb B}}

\newcommand{\esf}{\mbox{\bb S}}
\newcommand{\Te}{\mbox{\bb T}}
\newcommand{\Nsmall}{\mbox{\bbsmall N}}
\newcommand{\Rsmall}{\mbox{\bbsmall R}}

\newcommand{\rth}{\R^3}

\newcommand{\Mint}{M_{\infty}}
\newcommand{\TM}{{\cal T}(M)}
\newcommand{\TS}{{\cal T}(\S)}

\def\De{{\Delta}}
\def\S{{\Sigma}}
\def\s{{\sigma}}
\def\a{{\alpha}}

\def\g{{\gamma}}
\def\G{{\Gamma}}

\def\de{{\delta}}

\def\ve{{\varepsilon}}

\def\centerbmp#1#2#3{\vskip#2\relax\centerline{\hbox to#1{\special
    {bmp:#3 x=#1, y=#2}\hfil}}}

\usepackage[margin=1.25in]{geometry}

\newtheorem{theorem}{Theorem}[section]
\newtheorem{lemma}[theorem]{Lemma}

\newtheorem{remark}[theorem]{Remark}
\newtheorem{corollary}[theorem]{Corollary}
\newtheorem{definition}[theorem]{Definition}
\newtheorem{conjecture}[theorem]{Conjecture}
\newtheorem{assertion}[theorem]{Assertion}

\newenvironment{proof}{\smallskip\noindent{\it Proof.}\hskip \labelsep}
                          {\hfill\penalty10000\raisebox{-.09em}{$\Box$}\par\medskip}

\begin{document}
\begin{title}
{The Dynamics Theorem  for $CMC$ surfaces in $R^3$}
\end{title}
\begin{author}
{William H. Meeks, III\thanks{ This material
is based upon work for the NSF under Award No. DMS - 0703213. Any
opinions, findings, and conclusions or recommendations expressed in
this publication are those of the authors and do not necessarily
reflect the views of the NSF.} \and Giuseppe Tinaglia}
\end{author}
\maketitle
\begin{abstract}

In this paper, we study the space of translational limits ${\cal
T}(M)$ of a surface $M$ properly embedded in $\rth$ with nonzero
constant mean curvature and bounded second fundamental form. There
is a natural map ${\cal T}$ which assigns to any surface $\Sigma \in
{\cal T}(M)$, the set ${\cal T}(\Sigma)\subset {\cal T}(M)$. Among
various dynamics type results we prove that surfaces in minimal
${\cal T}$-invariant sets of ${\cal T}(M)$ are chord-arc. We also
show that if $M$ has an infinite number of ends, then there exists a
nonempty minimal ${\cal T}$-invariant set in ${\cal T}(M)$
consisting entirely of surfaces with planes of Alexandrov symmetry.
Finally, when $M$ has a plane of Alexandrov symmetry, we prove the
following characterization theorem: $M$ has finite topology if and
only if $M$ has a finite number of ends greater than one.


\vspace{.1cm} \noindent{\it Mathematics Subject Classification:}
Primary 53A10, Secondary 49Q05, 53C42

\end{abstract}

\section{Introduction.}

A general problem in classical surface theory is to describe the
asymptotic geometric structure of a connected, noncompact, properly
embedded, nonzero constant mean curvature ($CMC$) surface $M$ in
$\rth$. In this paper, we will show that when $M$ has bounded second
fundamental form, for any divergent sequence of points $p_n\in M$, a
subsequence of the translated surfaces $M-p_n$ converges to a
properly immersed surface of the same constant mean curvature
which bounds a smooth open subdomain on its mean convex side. The
collection $\Te(M)$ of all these limit surfaces sheds light on the
geometry of $M$ at infinity.

We will focus our attention on the subset ${\cal T}(M)\subset
\Te(M)$ consisting of the connected components of surfaces in
$\Te(M)$ which pass through the origin in $\rth$. Given a surface
$\S\in {\cal T}(M)$,  we will prove that ${\cal T}(\S)$ is always a
subset of ${\cal T}(M)$. In particular, we can consider ${\cal T}$ to represent a function:
$${\cal T}\colon {\cal T}(M) \to {\cal P}({\cal T}(M)),$$
where ${\cal P}({\cal T}(M))$ denotes the power set of ${\cal
T}(M)$. Using the fact that ${\cal T}(M)$ has a natural compact
metric space topology, we obtain classical dynamics type results on
${\cal T}(M)$ with respect to the mapping ${\cal T}$. These
dynamics results include the existence of nonempty minimal ${\cal
T}$-invariant sets in $\TM$ and are described in Theorem~\ref{T},
which we refer to as the $CMC$ Dynamics Theorem in $\rth$, or more
simply as just the Dynamics Theorem.

Assume $M\subset \rth$ is a connected, noncompact, properly embedded
$CMC$ surface with bounded second fundamental form. In section 3, we
demonstrate various properties of the minimal ${\cal T}$-invariant
sets in $\TM$. For example, we prove:
\begin{quote}
{\it Surfaces in minimal ${\cal T}$-invariant sets in $\TM$ are
chord-arc.}
\end{quote}
\begin{quote} {\it If $M$ has an infinite number of ends, then $\TM$
contains a minimal ${\cal T}$-invariant set in which every element
has a plane of Alexandrov symmetry.}
\end{quote}
\begin{quote}{\it If $M$ has finite genus, then any element
in a minimal ${\cal T}$-invariant set is a Delaunay
surface\footnote{In this manuscript, ``Delaunay surfaces'' refers to
the embedded $CMC$ surfaces of revolution discovered by
Delaunay~\cite{de1} in 1841.}.}\end{quote} In the special case that
$M$ has finite topology\footnote{A surface has {\it finite topology}
if it is homeomorphic to a closed surface minus a finite number of
points.}, this last result follows from the main theorem
in~\cite{kks1}, however the full generality of this result is needed
in applications in~\cite{mt1,mt2}.

In section~\ref{sc4}, we deal with  $CMC$ surfaces with a plane of
Alexandrov symmetry. In particular we obtain the following
characterization result:

\begin{quote} {\it If $M$ is a complete, connected, noncompact embedded $CMC$ surface with a plane of Alexandrov symmetry and bounded second
fundamental form, then $M$ has finite topology if and only if it has
a finite number of ends greater than one.}\end{quote}

The collection of properly embedded $CMC$ surfaces with bounded
second fundamental form is quite large and varied (see~\cite{gb1,
kap1, ka5, la3, map, mpp1}). Many of these examples
appear as doubly and singly-periodic surfaces. The
techniques of Kapouleas~\cite{kap1} and Mazzeo-Pacard~\cite{map} can
be applied to obtain many nonperiodic examples of finite and
infinite topology. Some theoretical aspects of the study of these
special surfaces have been developed previously in works of Meeks
~\cite{me17}, Korevaar-Kusner-Solomon~\cite{kks1} and
Korevaar-Kusner~\cite{kk2}; results from all of these three key
papers are applied here. More generally, the broader theory of
properly embedded $CMC$ surfaces in homogeneous three-manifolds is
an active field of research with many interesting recent results
~\cite{dh1, fm1, hars1}. In~\cite{mt5}, we will generalize the ideas
contained in this paper to obtain related theoretical results for
properly embedded separating $CMC$ hypersurfaces of bounded second
fundamental form in homogeneous $n$-manifolds.

In subsequent papers,~\cite{mt1,mt2}, we apply the results contained
in this manuscript. In~\cite{mt2}, we prove that the existence of a
Delaunay surface in $\TM$ implies $M$ does not admit any other
noncongruent isometric immersion into $\rth$ with the same constant
mean curvature (see also~\cite{ku2,smyt1}). In~\cite{mt1}, we show
that {\it any complete, embedded, noncompact, simply-connected $CMC$
surface $M$ in a fixed homogeneous three-manifold $N$ has the
appearance of a suitably scaled helicoid nearby any point of $M$
where the second fundamental form is sufficiently large}
(see~\cite{tin1} for a related result).

\vspace{.2cm} \noindent{Acknowledgements:} We thank Rob Kusner and
Joaquin Perez for their helpful comments on the results and proofs
contained in this paper.  We also thank Joaquin Perez for making the
figures that appear here.

\section{The Dynamics Theorem for $CMC$ surfaces of bounded
curvature.} \label{improved}

In this section, motivated by  previous work of Meeks, Perez and Ros
in~\cite{mpr10}, we prove a dynamics type result for the space
${\cal T}(M)$ of certain translational limit surfaces of a properly
embedded, $CMC$ surface $M\subset \rth$ with bounded second
fundamental form. All of these limit surfaces satisfy the
almost-embedded property described in the next definition.

\begin{definition} \label{def} {\rm Suppose $W$ is a complete flat three-manifold
with boundary $\partial W=\S$ together with an isometric immersion
$f\colon W \to \rth$ such that $f$ restricted to the interior of $W$
is injective. This being the case, if $f(\S)$ is a $CMC$ surface and
$f(W)$ lies on the mean convex side of $f(\S)$, we call the image
surface $f(\S)$ a {\em strongly Alexandrov embedded $CMC$ surface}.}
\end{definition}

We note that, by elementary separation properties, any properly
embedded $CMC$ surface  in $\rth$ is always strongly Alexandrov
embedded. Furthermore, by item~{\it 1} of Theorem~\ref{T} below, any
strongly Alexandrov embedded $CMC$ surface in $\rth$ with bounded
second fundamental form is properly immersed in $\rth$.

Recall that the only compact Alexandrov embedded\footnote{A compact
surface $\S$ immersed in $\rth$ is {\it Alexandrov embedded} if $\S$
is the boundary of a compact three-manifold immersed in $\rth$.}
$CMC$ surfaces in $\rth$ are spheres by the classical result of
Alexandrov~\cite{aa1}. Hence, from this point on, we will only
consider surfaces $M$ which are noncompact and connected.

\begin{definition} {\rm Suppose $M\subset \rth$ is a connected,
noncompact, strongly Alexandrov embedded $CMC$ surface with bounded
second fundamental form.
\begin{enumerate}
\item ${\cal T}(M)$ is the set of all connected, strongly
Alexandrov embedded $CMC$ surfaces $\S \subset \rth$, which are
obtained in the following way.

There exists a sequence of points $p_n\in M$, $\lim_{n\to
\infty}|p_n|=\infty$, such that the translated surfaces $M-p_n$
converge $C^2$ on compact sets of $\rth$ to a strongly Alexandrov
embedded $CMC$ surface $\Sigma'$, and $\Sigma$ is a connected
component of $\Sigma'$ passing through the origin. Actually we
consider the immersed surfaces in ${\cal T}(M)$ to be {\it pointed}
in the sense that if such a surface is not embedded at the origin,
then we consider the surface to represent two different elements in
${\cal T}(M)$ depending on a choice of one of the two preimages of
the origin.
\item $\Delta \subset  {\cal T}(M)$ is called
{\em ${\cal T}$-invariant}, if $\S\in\Delta$
implies ${\cal T}(\S)\subset \Delta$.
\item A nonempty subset $\Delta\subset {\cal T}(M)$ is
called a {\em minimal} ${\cal T}$-invariant
set, if it is ${\cal T}$-invariant and contains no smaller nonempty ${\cal
T}$-invariant sets.
\item If $\S \in {\cal T}(M)$ and $\S$ lies in a minimal ${\cal T}$-invariant
 set of ${\cal T}(M)$, then $\S$ is called a  {\em minimal
element} of ${\cal T}(M)$.
\end{enumerate}}
\end{definition}

Throughout the remainder of this paper, $\B(p,R)$ denotes the open
ball in $\rth$ of radius $R$ centered at the point $p$ and  $\B(R)$
denotes the open ball of radius $R$ centered at the origin in
$\rth$. Furthermore, we will always orient surfaces so that
their mean curvature $H$ is positive.

With these definitions in hand, we now state our Dynamics Theorem.

\begin{theorem}[$CMC$ Dynamics Theorem]\label{T}
Let $M\subset \rth$ be a connected, noncompact, strongly Alexandrov
embedded $CMC$ surface with bounded second fundamental form. Let $W$
be the associated complete flat three-manifold on the mean convex
side of $M$. Then the following statements hold:
\begin{enumerate}
\item \label{n1} $M$ is properly immersed in $\rth$.
\item \label{n2} There exist positive constants $c_1,c_2$ depending
only on the mean curvature of $M$ and on an upper bound for the norm
of its second fundamental form, such that for any $p\in M$ and
$R\geq 1$, \begin{equation}\label{eq4}c_1\leq \frac{{\rm
Area}(M\cap\B(p,R))}{{\rm Volume}(W\cap \B(p,R))}\leq
c_2.\end{equation} In particular, for $R\geq 1$, $\mbox{\rm
Area}(M\cap\B(R))\leq \frac{4\pi c_2}{3} R^3$. Furthermore, $M$ has
a regular neighborhood of radius $\ve$ in $W$, where $\ve>0$ only
depends on the mean curvature of $M$ and on an upper bound for the
norm of its second fundamental form.
\item \label{n3} $W$ is a handlebody\footnote{A {\it handlebody} is a three-manifold with
boundary which is homeomorphic to a closed regular neighborhood of
some  connected, properly embedded simplicial one-complex in
$\rth$.} and every point in $W$ is a distance of less than $\frac1H$
from $\partial W$, where $H$ is the mean curvature of $M$.
\item \label{n4} ${\cal T}(M)$ is nonempty and ${\cal T}$-invariant.
\item \label{n5} ${\cal T}(M)$ has a natural compact topological space structure
given by a metric $d_{{\cal T}(M)}$. The metric $d_{\TM}$ is induced
by the Hausdorff distance between compact subsets of $\rth$.
\item \label{n6} If $M$ is an element of ${\cal T}(M)$, then ${\cal T}(M)$ is a
connected space. In particular, if $M$ is invariant under a
translation, then ${\cal T}(M)$ is connected.
\item \label{n7} A nonempty  set $\Delta \subset {\cal T}(M)$ is a
minimal ${\cal T}$-invariant set if and
only if whenever $\S \in \Delta$, then ${\cal T}(\S)=\Delta$.
\item \label{n8} Every nonempty ${\cal T}$-invariant
set of ${\cal T}(M)$ contains a nonempty minimal ${\cal
T}$-invariant set. In particular, since ${\cal T}(M)$ is itself a
nonempty ${\cal T}$-invariant set, ${\cal T}(M)$ always contains
nonempty minimal invariant sets.
\item \label{n9} Any minimal ${\cal T}$-invariant set in ${\cal T}(M)$ is a compact
connected subspace of ${\cal T}(M)$.

\end{enumerate}
\end{theorem}

\begin{proof}
For the proofs of items~{\it \ref{n1}} and {\it \ref{n2}} see
Corollary~5.2 in~\cite{mt3} or see~\cite{mr7}. The key idea in the
proof of Corollary~5.2 is to show that the immersed surface $M$ has
a fixed size regular neighborhood on its mean convex side.

We now prove item~{\it \ref{n3}}. The proof that $W$ is a handlebody
is based on topological techniques used previously to study the
topology of a complete, orientable flat three-manifold $X$ with
minimal surfaces as boundary. These techniques were first developed
by Frohman and Meeks~\cite{fme1} and later generalized by
Freedman~\cite{fre1}. An important consequence of the results and
theory developed in these papers is that if $\partial X$ is mean
convex, $X$ is not a handlebody, and $X$ is not a Riemannian product
of a flat surface with an interval, then $X$ contains an orientable,
noncompact, embedded, stable minimal surface $\S$ with compact
boundary. Suppose now that $M\subset \rth$ is a strongly Alexandrov
embedded $CMC$ surface with associated domain $W$ on its mean convex
side. Since $M$ is not totally geodesic, $W$ cannot be a Riemannian
product of a flat surface with an interval. Therefore, if $W$ is not
a handlebody, there exists an orientable, noncompact, embedded
stable minimal surface $\S\subset W$ with compact boundary. Since
$\S$ is orientable and stable, a result of Fisher-Colbrie~\cite{fi1}
implies $\S$ has finite total curvature. It is well known that such
a $\S$ has an end $E$ which is asymptotic to an end of a catenoid or
a plane ~\cite{sc1}. We will obtain a contradiction when $E$ is a
catenoidal type end; the case where $E$ is a planar type end can be
treated in the same manner. After a rotation of $M$, assume that the
catenoid to which $E$ is asymptotic is vertical  and $E$ is a graph
over the complement of a disk in the $(x_1,x_2)$-plane; assume the
disk is $\B(R)\cap \{x_3=0\}$ for some large $R$. Let $S^2$ be a
sphere in $\rth$ with mean curvature equal to the mean curvature of
$M$, which lies below $E$ and which is disjoint from the solid
cylinder $\{(x_1,x_2,x_3)\mid x_1^2+x_2^2\leq R^2\}$. By vertically
translating $S^2$ upward across the $(x_1,x_2)$-plane and applying
the maximum principle for $CMC$ surfaces, we find that as $S^2$
translates across $E$, the portions of the translated sphere that
lie above $E$ do not intersect $M=\partial W$. Thus, some vertical
translate $\h{S}^2$ of $S^2$ lies inside $W$. Next translate
$\h{S}^2$ inside $W$ so that it touches $\partial W$ a first time.
The usual application of the maximum principle for $CMC$ surfaces
implies that $M$ is a sphere, which is not possible since $M$ is not
compact.

Note that if some point $p\in W$ had distance at least $\frac1H$
from $\partial W$, then $\partial \B (p,\frac1H)$ is a sphere  of
mean curvature $H$ in $W$. The arguments in the previous paragraph
show that no such sphere can exist, and this contradiction completes
the proof of item~{\it \ref{n3}}.

The uniform local area estimates for $M$ given in item~{\it
\ref{n2}} and the assumed bound on the second fundamental form of
$M$, together with standard compactness arguments, imply that for
any divergent sequence of points $\{p_n\}_n$ in $M$, a subsequence
of the translated surfaces $M-p_n$ converges on compact sets  of
$\rth$ to a strongly Alexandrov embedded $CMC$ surface $\M_{\infty}$
in $\rth$. The component $M_{\infty}$ of $\M_{\infty}$ passing
through the origin is a surface in ${\cal T}(M)$ (if $M_{\infty}$ is
not embedded at the origin, then one obtains two elements in ${\cal
T}(M)$ depending on a choice of one of the two pointed components).
Hence, $\TM$ is nonempty.

Let $\S\in \TM$ and $\S'\in {\cal T}(\S)$. By definition of $\TS$,
any compact domain of $\S'$ can be approximated arbitrarily well by
translations of compact domains ``at infinity'' in $\S$. In turn, by
definition of $\TM$, these compact domains ``at infinity'' in $\S$
can be approximated arbitrarily well by translated compact domains
``at infinity'' on $M$. Hence, a standard diagonal argument implies
that $\S'\in\TM$. Thus, $\TM$ is ${\cal T}$-invariant, which proves
 item~{\it \ref{n4}}.

Suppose now that $\S \in {\cal T}(M)$ is embedded at the origin. In this
case, there exists an $\ve>0$ depending only on the bound of the
second fundamental form of $M$, so that there exists a disk
$D(\S)\subset \S\cap \overline{\B}(\ve)$ with $\partial
D(\S)\subset\partial\overline{\B}(\ve)$, $\vec{0}=(0,0,0) \in D(\S)$
and such that $D(\S)$ is a graph with gradient at most 1 over its
projection to the tangent plane $T_{\vec{0}}D(\S)\subset \rth$.
Given another such $\S'\in {\cal T}(M)$, define
$$ d_{{\cal T}(M)}(\S,\S')=d_{\cal H}(D(\S),D(\S')),$$
where $d_{\cal H}$ is the Hausdorff distance. If $\vec{0}$ is not a
point where $\S$ is embedded, then since we consider $\S$ to
represent one of two different pointed surfaces in ${\cal T}(M)$, we
choose $D(\S)$ to be the disk in $\S\cap\B(\ve)$ containing the
chosen base point. With this modification, the above metric is
well-defined on ${\cal T}(M)$.

Using the fact that the surfaces in ${\cal T}(M)$ have uniform local
area and curvature estimates (see item~{\it \ref{n2}}), we will now
prove ${\cal T}(M)$ is sequentially compact and hence compact. Let
$\{\S_n\}_n$ be a sequence of surfaces in ${\cal T}(M)$ and let
$\{D(\S_n)\}_n$ be the related sequence of graphical disks defined
in the previous paragraph. A standard compactness argument implies
that a subsequence, $\{D(\S_{n_i})\}_{n_i}$ of these disks converges
to a graphical $CMC$ disk $D_\infty$. Using item~{\it \ref{n2}}, it
is straightforward to show that $D_\infty$ lies on a complete,
strongly Alexandrov embedded surface $\S_\infty$ with the same
constant mean curvature as $M$. Furthermore, $\S_\infty$ is a limit
of compact domains $\Delta_{n_i}\subset \S_{n_i}$. In turn, the
$\Delta_{n_i}$'s are limits of translations of compact domains in
$M$, where the translations diverge to infinity. Hence, $\S_\infty$
is in ${\cal T}(M)$ and by definition of $d_{{\cal T}(M)}$, a
subsequence of $\{\S_n\}_n$ converges to $\S_\infty$. Thus, ${\cal
T}(M)$ is a compact metric space with respect to the metric
$d_{{\cal T}(M)}$. We remark that this compactness argument can be
easily modified to prove that the topology of  ${\cal T}(M)$ is
independent of the sufficiently small radius $\ve$ used to define
$d_{{\cal T}(M)}$. It follows that the topological structure on
${\cal T}(M)$ is determined  ($\ve$ chosen sufficiently small), and
it is in this sense that the topological structure is natural. This
completes the proof of item~{\it \ref{n5}}.

Suppose now that $M\in{\cal T}(M)$. Note that whenever $X\in {\cal
T}(M)$, then the path connected set of translates ${\rm
Trans}(X)=\{X-q\mid q\in X\}$ is a subset of ${\cal T}(M)$. In
particular,  ${\rm Trans}(M)$ is a subset of ${\cal T}(M)$. We claim
that the closure of ${\rm Trans}(M)$ in ${\cal T}(M)$ is equal to
${\cal T}(M)$. By definition of closure, the closure of ${\rm
Trans}(M)$  is a subset of $ {\cal T}(M)$. Using the definition of
${\cal T}(M)$ and the metric space structure on ${\cal T}(M)$, it is
straightforward to check that ${\cal T}(M)$ is contained in the
closure of ${\rm Trans}(M)$; hence, $\overline{\rm Trans(M)}=\TM$.
Since the closure of a path connected set in a topological space is
always connected, we conclude that ${\cal T}(M)$ is connected, which
completes the proof of item~{\it \ref{n6}}.

We now prove item~{\it \ref{n7}}. Suppose $\Delta$ is a nonempty,
minimal ${\cal T}$-invariant set and $\S \in \De$. By definition of
${\cal T}$-invariance, $\TS\subset\De$. By item~{\it \ref{n4}},
$\TS$ is a nonempty ${\cal T}$-invariant set. By definition of
minimal ${\cal T}$-invariant set, $\TS=\De$, which proves one of the
desired implications. Suppose now that $\De\subset \TS$ is nonempty
and that whenever $\S\in \De$, $\TS=\De$; it follows that $\De$ is a
${\cal T}$-invariant set. If $\De'\subset \De$ is a nonempty ${\cal
T}$-invariant set, then there exists a $\S'\in \De'$, and thus,
$\De={\cal T}(\S')\subset\De'\subset \De$. Hence, $\De'=\De$, which
means $\De$ is a minimal ${\cal T}$-invariant set and item~{\it
\ref{n7}} is proved.

Now we prove item~{\it \ref{n8}} through an application of Zorn's
lemma. Suppose $\Delta \subset {\cal T}(M)$ is a nonempty ${\cal
T}$-invariant set and $\S \in \Delta$. Using the definition of
${\cal T}$-invariance, an elementary argument proves ${\cal
T}(\Sigma )$ is a nonempty ${\cal T}$-invariant set in $\Delta$
which is a closed set of ${\cal T}(M)$; essentially, this is because
the set of limit points of a set in a topological space forms a
closed set (also see the proofs of items~{\it \ref{n4}} and {\it
\ref{n5}} for this type of argument). Next consider the set $\Lambda
$ of all nonempty ${\cal T}$-invariant subsets of $\Delta $ which
are closed sets in ${\cal T}(M)$, and as we just observed, this
collection is nonempty. Also, observe that $\Lambda $ has a partial
ordering induced by inclusion $\subset$.

We first check that any linearly ordered set in $\Lambda $ has a
lower bound, and then apply Zorn's Lemma to obtain a minimal element
of $\Lambda$ with respect to the partial ordering $\subset$. To do
this, suppose $\Lambda '\subset \Lambda $ is a nonempty linearly
ordered subset and we will prove that the intersection $\bigcap
_{\Delta'\in \Lambda '}\Delta '$ is an element of $\Lambda $. In our
case, this means that we only need to prove that such an
intersection is nonempty, because the intersection of closed
(respectively ${\cal T}$-invariant) sets in a topological space is a
closed set (respectively ${\cal T}$-invariant) set. Since each
element of $\Lambda'$ is a closed set of ${\cal T}(M)$ and the
finite intersection property holds for the collection $\Lambda'$,
then the compactness of $\TM$ implies $\bigcap _{\Delta'\in \Lambda
'}\Delta ' \neq \mbox{\O}$. Thus, $\bigcap _{\Delta'\in \Lambda
'}\Delta ' \in \Lambda$ is a lower bound for $\Lambda'$. By Zorn's
lemma applied to $\Lambda$ under the partial ordering $\subset$,
$\Delta$ contains a smallest, nonempty, closed ${\cal T}$-invariant
 set $\Omega$. We now check that $\Omega$ is a nonempty, minimal
${\cal T}$-invariant subset of $\Delta$. If $\Omega'$ is a nonempty
${\cal T}$-invariant subset of $\Omega$, then there exists a
$\Sigma'\in\Omega'$. By our previous arguments, ${\cal
T}(\Sigma')\subset \Omega'\subset \Omega$ is a nonempty ${\cal
T}$-invariant set in $\Delta$ which is a closed set in ${\cal
T}(M)$, i.e., ${\cal T}(\S')\in \Lambda$. Hence, by the minimality
property of $\Omega$ in $\Lambda$, we have ${\cal T}(\Sigma')=
\Omega'=\Omega$. Thus, $\Omega$ is a nonempty, minimal ${\cal
T}$-invariant subset of $\Delta$, which proves item~{\it \ref{n8}}.

Let $\Delta\subset {\cal T}(M)$ be a nonempty, minimal ${\cal
T}$-invariant set and let $\S\in \Delta$. By item~{\it \ref{n7}},
${\cal T}(\S)=\Delta$. Since ${\cal T}(\Sigma)$ is a closed set in
$\TM$ and ${\cal T}(M)$ is compact, then $\Delta$ is compact. Since
$\S\in {\cal T}(\S)=\Delta$, item~{\it \ref{n6}} implies $\Delta$ is
also connected which completes the proof of item~{\it \ref{n9}}.
\end{proof}

\begin{remark}\label{rm25}
{\rm It turns out that any complete, connected, noncompact, embedded
$CMC$ surface $M\subset\rth$ with compact boundary and bounded
second fundamental form, is properly embedded in $\rth$, has a fixed
sized regular neighborhood on its mean convex side and so has
cubical area growth; these properties of $M$ follow from simple
modifications of the proof of these properties in the case when $M$
has empty boundary (see~\cite{mr7,mt3}). For such an $M$, the space
${\cal T}(M)$ also can be defined and consists of a nonempty set of
strongly Alexandrov embedded $CMC$ surfaces without boundary. We
will use this remark in the next section where $M$ is allowed to
have compact boundary. Also we note that items~{\it \ref{n4}} - {\it
\ref{n9}} of the Dynamics Theorem make sense under small
modifications and hold for properly embedded separating $CMC$
hypersurfaces $M$ with bounded second fundamental form in noncompact
homogeneous $n$-manifolds $N$, where ${\cal T}(M)$ is the set of
connected properly immersed surfaces that pass through a fixed base
point of $N$ and which are components of limits of $M$ under a
sequence of ``translational" isometries of $N$ which take a
divergent sequence of points in $M$ to the base point; see
~\cite{mt5} for details. }
\end{remark}

\section{The Minimal Element Theorem.}

In this section, we give applications of the Dynamics Theorem to the
theory of complete embedded $CMC$ surfaces $M$ in $\rth$ with
bounded second fundamental form and compact boundary. Let $R$ be the
radial distance to the origin in $\rth$. We will obtain several
results concerning the geometry of minimal elements in ${\cal
T}(M)$, when the area growth of $M$ is less than cubical in $R$ or
when the genus of the surfaces $M\cap \B(R)$ grows slower than
cubically in $R$. With this in mind, we now define some growth
constants for the area and genus of $M$ in $\rth$.

For any $p\in M$, we denote by $M(p,R)$ the connected component of
$M\cap \B(p,R)$ which contains $p$; if $M$ is not embedded at $p$
and there are two immersed components $M(p,R)$, $M'(p,R)$
corresponding to two pointed immersions, then in what follows we
will consider both of these components separately.

\begin{definition}[Growth Constants]
For $n=1,2,3,$ we define:
$$A_{\sup}(M,n)=\limsup \sup_{p\in M}({\rm Area}[M(p,R)]\cdot R^{-n}),$$
$$A_{\inf}(M,n)=\liminf \inf_{p\in M}({\rm Area}[M(p,R)]\cdot R^{-n}),$$
$$G_{\sup}(M,n)=\limsup \sup_{p\in M}({\rm Genus}[M(p,R)]\cdot R^{-n}),$$
$$G_{\inf}(M,n)=\liminf \inf_{p\in M}({\rm Genus}[M(p,R)]\cdot R^{-n}).$$
\end{definition}

In the above definition, note that $\sup_{p\in M}({\rm
Area}[M(p,R)]\cdot R^{-n})$ and the other similar expressions are
functions from $(0,\infty)$ to $\R$ and therefore they each have a
$\limsup$ or a $\liminf$, respectively.

By item~{\it \ref{n2}}\, of Theorem~\ref{T} and Remark~\ref{rm25},
$A_{\sup}(M,3)$ is a finite number. We now prove that
$G_{\sup}(M,3)$ is also finite. Since $M$ has  bounded second
fundamental form, it admits a triangulation $T$ whose edges are
geodesic arcs or smooth arcs in the boundary of $M$ of lengths
bounded between two small positive numbers, and so that the areas of
2-simplices in $T$ also are bounded between two small positive
numbers.   Let $T(M(p,R))$ be the set of simplices in $T$ which
intersect $M(p,R)$.  Note that for $R$ large, the number of edges in
$T(M(p,R)) $ which intersect $M(p,R)$ is less than some constant $K$
times the area of $M(p,R)$, where $K$ depends only on the second
fundamental form of $M$. Hence, the number of generators of the
first homology group $H_1(T( M(p,R)),\R)$ is less than $K$ times the
area of $M(p,R)$.  Since there are at least $ {\rm Genus}[M(p,R)]$
linearly independent simplicial homology classes
 in $H_1(T( M(p,R)),\R)$, then
\begin{equation}\label{eq5} {\rm Genus}[M(p,R)] \leq K {\rm Area}[M(p,R)]
\quad \text{for } R \text{ large}.\end{equation} In particular,
since $A_{\sup}(M,3)$ is finite, equation \eqref{eq5} implies that
$G_{\sup}(M,3)$ is also finite.

\begin{definition} {\rm Suppose that $M\subset \rth$ is a complete,
noncompact, connected embedded $CMC$ surface with compact boundary
(possibly empty) and with bounded second fundamental form.

\begin{enumerate} \item For any divergent sequence of
points $p_n\in M$, a subsequence of the translated surfaces $M-p_n$
converges to a properly immersed surface  of the same constant mean
curvature which bounds a smooth open subdomain on its mean convex
side. {\it Let $\Te(M)$ denote the collection of all such limit
surfaces.}

\item If there exists a constant $C>0$ such that for all $p,q\in
M$ with $d_{\Rsmall^3}(p,q)\geq 1$, $d_M(p,q) \leq C \cdot
d_{\Rsmall^3}(p,q)$, then we say that $M$ is {\it chord-arc.} (Note
that the triangle inequality implies that if $M$ is chord-arc and
$p,q\in M$ with $d_{\Rsmall^3}(p,q)<1$, then $d_{M}(p,q)<6C$.)
 \end{enumerate}

 }
 \end{definition}

We note that in the above definition and in Theorem~\ref{sp2} below,
the embedded hypothesis on $M$ can be replaced by the weaker
hypothesis that $M$ has a fixed size one-sided neighborhood on its
mean convex side (see Remark~\ref{rm25}).

We now state the main theorem of this section. For the statement of
this theorem, recall that a plane $P\subset \rth$ is a {\it plane of
Alexandrov symmetry} for a surface $M\subset \rth$, if it is a plane
of symmetry which separates $M$ into two open components $M^+,$
$M^-,$ each of which is a graph over a fixed subdomain of $P$.

\begin{theorem}[Minimal Element Theorem] \label{sp2} Let $M\subset
\rth$ be a complete,  noncompact, connected embedded $CMC$ surface
with possibly empty compact boundary and bounded second fundamental
form. Then the following statements hold.

\begin{enumerate}
\item \label{oneend} If $\S\in \TM$ is a minimal element, then
either every surface in $\Te(\S)$ is the translation of a fixed
Delaunay surface or every surface in $\Te(\S)$ has one end. In
particular, if $\S\in \TM$ is a minimal element, then every surface
in $\Te(\S)$ is connected and $\TS=\Te(\S)$.
\item \label{sca} Minimal elements of $\TM$ are chord-arc.
\item \label{n10} Let $\S$ be a minimal element of ${\cal T}(M)$.
For all $D, \, \ve>0$, there exists a $d_{\ve,D}>0$ such that the
following statement holds. For every compact domain $X\subset \S$
with extrinsic diameter less than $D$ and for each $q\in \S$, there
exists a smooth compact, domain $X_{q,\ve}\subset \S$ and a
translation, $\tau\colon \rth \to \rth$, such that
$$d_{\S}(q,X_{q,\ve})<d_{\ve,D}\;\;\; \mbox{and} \;\;\; d_{\cal H}(X,
\tau(X_{q,\ve}))<\ve,$$ where $d_{\S}$ is the distance function on
$\S$ and $d_{\cal H}$ is the Hausdorff distance on compact sets in
$\rth$. Furthermore, if $X$ is connected, then $X_{q,\ve}$ can be
chosen to be connected.
\item \label{half} If $M$ has empty boundary and lies in the halfspace
$\{x_3\geq 0\}$, then some minimal element of ${\cal T}(M)$ has the
$(x_1,x_2)$-plane as a plane of Alexandrov symmetry.
\item \label{balls} If $E$ is an end
representative\footnote{A proper noncompact domain $E\subset M$ is
called an {\it end representative} for $M$ if it is connected and
has compact boundary.} of $M$ such that $\rth - E$ contains  balls
of arbitrarily large radius, then ${\cal T}(M)$ contains a surface
with a plane of Alexandrov symmetry.
\item \label{inf3} The following statements are equivalent:
\begin{enumerate}
\item \label{A3} $A_{\inf}(M,3)=0.$
\item  \label{G3} $G_{\inf}(M,3)=0.$
\item  \label{sym} ${\cal T}(M)$ contains a minimal element with a plane of
Alexandrov symmetry.
\item  \label{Afinite2} $A_{\inf}(M,2)$ is finite.
\item  \label{Gfinite2} $G_{\inf}(M,2)$ is finite.
\end{enumerate}
\item \label{infty} If $M$ has an infinite number of ends, then  there exists a
minimal element in ${\cal T}(M)$ with a plane of Alexandrov
symmetry.
\item If ${\cal T}(M)$ does not contain an element with a plane of
Alexandrov symmetry, then the following statements hold.
\begin{enumerate}
\item \label{bal_a}  There exists a constant $F$ such that for every end
representative $E$ of a surface in $\Te(M)$, there exists a positive
number $R(E)$ such that
$$[\rth -\B(R(E))]\subset \{x\in \rth \mid d_{\Rsmall^3} (x,E)<F\}.$$
In particular, if $E_1$ and $E_2$ are end representatives of a
surface in $\Te(M)$, then for $R$ sufficiently large, for any $x\in
E_1 -\B(R)$, $d_{\Rsmall^3} (x,E_2-\B(R))<F\}$.
\item \label{bal_b} There is a uniform upper bound on the number of ends of
any element in $\Te(M)$. In particular, there is a uniform upper
bound on the number of components of any element in $\Te(M)$.
\end{enumerate}
\item \label{inf2} Suppose $\Sigma$ is a minimal element of ${\cal T}(M)$. Then
the following statements are equivalent.
\begin{enumerate}
\item  \label{A2} $A_{\inf}(\Sigma, 2)=0$.
\item  \label{G2} $G_{\inf}(\Sigma, 2)=0$.
\item  \label{D} $\Sigma$ is a Delaunay surface.
\item  \label{Afinite1} $A_{\inf}(\Sigma,1)$ is finite.
\item  \label{Gfinite1} $G_{\inf}(\Sigma,1)$ is finite.
\end{enumerate}

\end{enumerate}
\end{theorem}

The following corollary gives some immediate consequences of
Theorem~\ref{sp2}. The proof of this corollary appears after the
proof of Theorem~\ref{sp2}.

\begin{corollary} \label{cor2} Let $M\subset
\rth$ be a complete, noncompact, connected, embedded $CMC$ surface
with compact boundary and bounded second fundamental form. Then the
following statements hold.
\begin{enumerate}
\item $A_{\sup}(M,3)=0\quad \implies \quad G_{\sup}(M,3)=0
 \quad \implies $ \\ $\implies \quad$
Every minimal element in ${\cal T}(M)$ has a plane of Alexandrov
symmetry.
\item $A_{\sup}(M,2)=0 \quad \implies \quad G_{\sup}(M,2)=0
\quad \implies$ \\ $\implies \quad$
Every minimal element in ${\cal T}(M)$ is a Delaunay surface.
\end{enumerate}
\end{corollary}

We make the following conjecture related to the Minimal Element
Theorem. Note that item~{\it \ref{inf2}} of Theorem~\ref{sp2}
implies that the conjecture holds for $n=1$.

\begin{conjecture} Suppose that $M\subset \rth$ satisfies the
hypotheses of Theorem~\ref{sp2}.  Then for any minimal element $\S
\in {\cal T}(M)$ and  for $n= 1, \,2,$ or  $3,$ $$\lim_{R\to \infty}
 {\rm Area}[\S\cap \B(R)]\cdot R^{-n} \; \text{  and  } \; \lim_{R\to \infty}
 {\rm Genus}[\S\cap \B(R)]\cdot R^{-n}$$ exist (possibly infinite).
 Furthermore,
$$A_{\inf}(\S,n)=A_{\sup}(\S,n)=\lim_{R\to \infty}
 {\rm Area}[\S\cap \B(R)]\cdot R^{-n}$$
$$G_{\inf}(\S,n)=G_{\sup}(\S,n)=\lim_{R\to \infty}
 {\rm Genus}[\S\cap \B(R)]\cdot R^{-n}.$$

 \end{conjecture}

\vspace*{.3cm}\noindent{\it Proof of Theorem~\ref{sp2}.} We
postpone the proofs of items~{\it \ref{oneend}, \ref{sca},
\ref{n10}} to after the proofs of the items~{\it \ref{half} -
\ref{inf2}} of the theorem.

Assume  that $M$ has empty boundary and $M\subset \{x_3\geq 0\}$.
Using techniques similar to the ones discussed by Ros and Rosenberg
in~\cite{ror1}, we now prove that some element of $\TM$ has a
horizontal plane of Alexandrov symmetry, that is, item~{\it
\ref{half}}. Let $W_M$ be the smooth open domain in $\rth-M$ on the
mean convex side of $M$. Note that $W_M \subset \{ x_3\geq 0\}$.
After a vertical translation of $M$, assume that $M$ is not
contained in a smaller halfspace of $\{x_3\geq 0\}.$ Since $M$ has a
fixed size regular neighborhood on its mean convex side and $M$ has
bounded second fundamental form, then for any generic and
sufficiently small $\ve>0$, $M_{\ve}=M\cap \{x_3\leq \ve\}$ is a
nonempty graph of small gradient over its projection to
$P_0=\{x_3=0\}$; we let $P_t=\{x_3=t\}$. Note that the mean
curvature vector to $M_{\ve}$ is upward pointing. In what follows,
$R_{P_t}\colon \rth\to \rth$ denotes reflection in $P_t$, while
$\Pi\colon \rth\to \rth$ denotes orthogonal projection onto $P_0$.

For any $t>0$, consider the new surface with boundary,
$\widehat{M}_t$, obtained by reflecting $M_t=M\cap\{x_3\leq t\}$
across the plane $P_t$, i.e., $\h{M}_t=R_{P_t}(M_t)$. Let  $
T=\sup \{t\in (0,\infty)\mid {\rm for}\;\; t'<t$, the surface
$M_{t'}$ is a graph over its projection to $P_0$,
$\widehat{M}_{t'}\cap M=
\partial \widehat{M}_{t'}=\partial M_{t'}$ and the infimum of the
angles that the tangent spaces to $M$ along $\partial M_t$ make with
vertical planes is bounded away from zero\}. Recall that by height
estimates for $CMC$ graphs with zero boundary values~\cite{ror1},
$\ve<T\leq \frac{1}{H}$, where $H$ is the mean curvature of $M$.

If there is a point $p\in \partial M_T$ such that the tangent plane
$T_pM$ is vertical, then the classical Alexandrov reflection
principle implies that the plane $P_T$ is a plane of Alexandrov
symmetry. Next suppose that the angles that the tangent spaces to
$M_T$ make with $(0,0,1)$ along $\partial M_T$ are not bounded away
from zero. In this case, let $p_n\in \partial M_T$ be a sequence of
points such that the tangent planes $T_{p_n}M$ converge to the
vertical (the dot products of the normal vectors to the planes with
$(0,0,1)$ are going to zero) and let $\Sigma \in {\cal T}(M)$ be a
related limit of the translated surfaces $M-p_n$. One can check that
$\Sigma \cap \{x_3<0\} $ is a graph over $P_0$ and that its tangent
plane at the origin is vertical. Now the usual application of the
boundary Hopf maximum principle at the origin, or equivalently, the
Alexandrov reflection argument, implies $P_0$ is a plane of
Alexandrov symmetry for $\Sigma$.

Suppose now that the tangent planes of $M$ along $\partial M_T$ are
bounded away from the vertical. In this case, $P_T$ is not a plane
of Alexandrov symmetry. So, by the usual application of the
Alexandrov reflection principle, we conclude that $\widehat{M}_T
\cap M=\partial \widehat{M}_T=\partial M_T$. By definition of $T$,
there exist $\delta_n>0$, $\delta_n\rightarrow 0$, such that
$F_n=\widehat{M}_{T+\delta_n}\cap M$ is not contained in $\partial
M_{T+\delta_n}$. We first show that not only is $\Pi(F_n)$ contained
in the interior of $\Pi(M_T)$, but for some $\eta>0$, it stays at a
distance at least $\eta$ from $\Pi(\partial M_T)$ for $\delta_n$
sufficiently small. In fact, since we are assuming that the tangent
planes of $M$ along $\partial M_T$ are bounded away by a fixed
positive angle from the vertical, if $\delta$ is small enough, the
tangent planes of $M$ along $\partial M_{T+\delta}$ are also bounded
away by a fixed positive angle from the vertical. Thus, the previous
statement on the existence of an $\eta >0$ is a consequence of the
existence of a fixed size one-sided regular neighborhood for $M$ in
$W_M$.

The discussion in the previous paragraph implies that there exists a
sequence of points $p_n\in M_T$ which stay at a distance at least
$\eta$ from $\partial M_T$ and such that the distance from
$R_T(p_n)$ and $M-M_T$ is going to zero. The fact that $p_n$ stays
at a distance at least $\eta$ from $\partial M_T$ implies that for
$n$ large there exists an $\ve>0$ such that $R_T(\B(p_n,\ve)\cap M)$
is disjoint from $M$ and it is a graph over $\Pi(\B(p_n,\ve)\cap
M)$. Consider the element $\Sigma\in {\cal T}(M)$ obtained as a
limit of the translated surfaces $M-\Pi(p_n)$ and let $\lim_{n\to
\infty}p_n=p=(0,0,S)\in \S$. From the way $\Sigma$ is obtained, $p$
is a positive distance from $\partial \S_T$. Moreover, $R_T(p)\in
\S-\S_T$ and $\widehat{\S}_T$ is tangent to $\S-\S_T$ and lies on
its mean convex side. The maximum principle implies that $P_T$ is a
plane of Alexandrov symmetry which contradicts the assumption that
tangent planes of $M$ along $\partial M_T$ are bounded away by a
fixed positive angle from the vertical. This completes the proof
that there exists a surface $\Sigma \in \TM$ with the
$(x_1,x_2)$-plane as a plane of Alexandrov symmetry. It then follows
from item~{\it \ref{n8}} of Theorem~\ref{T} that the nonempty ${\cal
T}$-invariant set $\TS\subset \TM$ contains minimal element of $\TM$
with the $(x_1,x_2)$-plane as a plane of Alexandrov symmetry, which
proves item~{\it \ref{half}}.

We now prove item~{\it \ref{balls}} holds. Assume now that $M$ has
possibly nonempty compact boundary and there exists a sequence of
open balls $\B(q_n,n)\subset \rth-M$. Note that these balls can be
chosen so that they are at distance at least $n$ from the boundary
of $M$ and so that there exists a sequence of points $p_n\in\partial
\B(q_n,n)\cap M$ diverging in $\rth$. After choosing a subsequence,
we may assume that the translated balls $\B(q_n,n)-p_n$ converge to
an open halfspace $K$ of $\rth$ and a subsequence of the translated
surfaces $M-p_n$ gives rise to an element $M_\infty\in{\cal T}(M)$
with $M_\infty$ contained in the halfspace $\rth -K$ and $\partial
M_{\infty}=\mbox{\O}$. By the previous discussion when $M$ has empty
boundary (item {\it \ref{half}}), $ {\cal T}(M_\infty)\subset {\cal
T}(M)$ contains a minimal element with a plane of Alexandrov
symmetry. This completes the proof of item~{\it \ref{balls}}.

We now prove item~{\it \ref{inf3}}\, in the theorem. First observe
that ${\it  \ref{Afinite2} \implies~\ref{A3}}$ and that ${\it
\ref{Gfinite2}\implies~\ref{G3}}$. Also, equation \eqref{eq5}
implies that ${\it  \ref{A3}\implies~\ref{G3}}$ and that ${\it
\ref{Afinite2}\implies~\ref{Gfinite2}}$. We now prove that ${\it
\ref{sym}\implies~\ref{Afinite2}}$. Suppose that ${\cal T}(M)$
contains a minimal element $\Sigma$ which has a plane of Alexandrov
symmetry and let $W_{\Sigma}$ denote the embedded three-manifold on
the mean convex side of $\Sigma$. In this case $W_\S$ is contained
in a slab, and by item~{\it \ref{n2}} of Theorem~\ref{T}, the area
growth of $\Sigma$ is comparable to the volume growth of
$W_{\Sigma}$. Note that the volume of $W_{\Sigma}$ grows at most
like the volume of the slab which contains it, and so, the volume
growth of $W_{\Sigma}$ and the area growth of $\S$ is at most
quadratic in $R$. By the definitions of ${\cal T}(M)$ and
$A_{\inf}(M,2)$, we see that $A_{\inf}(M,2)$ is finite which
implies~{\it  \ref{Afinite2}}.

In order to complete the proof of item~{\it \ref{inf3}}, it suffices
to show ${\it  \ref{G3}\implies  \ref{sym}}$. However, since the
proof of ${\it \ref{G3}\implies  \ref{sym}}$ uses the fact that
${\it  \ref{A3}\implies  \ref{sym}}$, we first show that ${\it
\ref{A3}\implies  \ref{sym}}$. Assume that $A_{\inf}(M,3)=0$ and we
will prove that ${\cal T}(M)$ contains a surface $\Sigma$ which lies
in a halfspace of $\rth$. Since $A_{\inf}(M,3)=0$, we can find a
sequence of points $\{p_n\}_n\subset M$ and positive numbers $R_n$,
$R_n\to \infty$, such that the connected component $M(p_n, R_n)$ of
$M\cap \overline{\B} (p_n,R_n)$ containing $p_n$ has area less than
$\frac{1}{n}R_n^3.$ Since $M$ has bounded second fundamental form,
there exists an $\ve>0$ such that for any $q\in \rth$, if
$\B(q,r)\cap M \not =\mbox{\O}$, then ${\rm Area}(\B(q,r+1)\cap
M)\geq \ve$. Using this observation, together with the inequality
${\rm Area}(M\cap \overline{\B}(p_n,R_n))\leq \frac{1}{n} R_n^3$ and
the equality ${\rm Volume}\, (\B(p_n, R_n))= \frac{4\pi}{3} R_n^3$,
we can find a sequence of points $q_n\in \B(p_n, R_n)$, numbers
$k_n$ with $k_n\to \infty$, such that $\B(q_n, k_n)\subset [\B(p_n,
\frac{R_n}{2})-M(p_n,R_n)]$ and such that there are points $s_n\in
\partial \B(q_n,k_n)\cap M(p_n, R_n)$ with $|s_n|\to \infty$ (see Figure \ref{fig1b}).
Let $\Sigma \in {\cal T}(M)$ be a limit surface arising from the
sequence of translated surfaces  $M(p_n, R_n) - s_n$. Note that
$\Sigma$ is disjoint from an open halfspace obtained from a limit of
a subsequence of the translated balls $\B(q_n, k_n)-s_n$. Since
$\Sigma$ lies in a halfspace of $\rth$, item~{\it \ref{half}} in the
theorem implies ${\cal T}(M)$ contains a minimal element with a
plane of Alexandrov symmetry. The existence of this minimal element
proves that  ${\it  \ref{A3} \implies   \ref{sym}}$.

\begin{figure}[h!]\begin{center}
\includegraphics[width=2.7in]{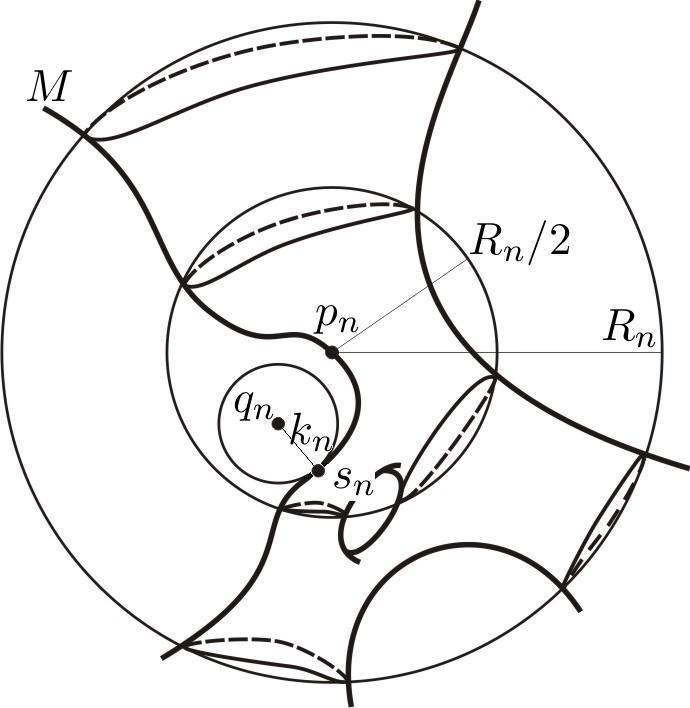}\end{center}
\caption{Finding large balls in the complement of $M(p_n,
R_n)$}\label{fig1b}
\end{figure}


We now prove that ${\it  \ref{G3} \implies  \ref{sym}}$ and this
will complete the proof of item~{\it \ref{inf3}}. Assume  that
$G_{\inf}(M,3)=0$. Since $G_{\inf}(M,3)=0$, there exists a sequence
of points $p_n\in M$ and $R_n \to \infty$, such that the genus of
$M(p_n, R_n) \subset \overline{\B}(p_n,R_n)$ is less than
$\frac{1}{n} R_n^3$. Using the fact that the genus of disjoint
surfaces is additive, a simple geometric argument, which is similar
to the argument that proved ${\it  \ref{A3} \implies  \ref{sym}}$,
shows that we can find a sequence of points $q_n\in \B(p_n, R_n)$
diverging in $\rth$ and numbers $k_n$, with $k_n\to \infty$, such
that one of the following statements holds.

\begin{enumerate}
\item $\mbox{Genus}(M(q_n, k_n))=0$.
\item $\B(q_n, k_n)\subset [\B(p_n, \frac{R_n}{2})-M(p_n,R_n)]$
and, as $n$ varies, there exist points $s_n\in
\partial \B(q_n,k_n)\cap M(p_n, R_n)$ diverging in $\rth$.
\end{enumerate}

If statement 2 holds, then our previous arguments imply that ${\cal
T}(M)$ contains a surface $\Sigma$ which lies in a halfspace of
$\rth$ and that ${\cal T}(M)$ contains a minimal element with a
plane of Alexandrov symmetry. Thus, we may assume statement 1 holds.

Since statement 1 holds, then the sequence of translated surfaces
$M-q_n$ yields a limit surface $\Sigma \in \TM$ of genus zero. If
$\S$ has a finite number of ends, then $\S$ has an annular end $E$.
By the main theorem in~\cite{me17}, $E$ is contained in a solid
cylinder in $\rth$. Under a sequence of translations of $E$, we
obtain a limit surface $D\in \TS$ which is contained in a solid
cylinder. By item~{\it \ref{half}}, there is a minimal element
$D'\in {\cal T}(D)\subset \TM$ which has a plane of Alexandrov
symmetry; this conclusion also follows from the main result
in~\cite{kks1}.

Suppose now $\Sigma$ has genus zero and an infinite number of ends. For each $n\in
\N$, there exists numbers, $T_n$ with $T_n\to \infty$, such that the
number $k(n)$ of noncompact components, $$\{\S_1(T_n),
\S_2(T_n),\ldots,\S_{k(n)}(T_n)\},$$ in $\Sigma-\B(T_n)$ is at least
$n$. Fix points $p_i(n)\in \S_i(T_n)\cap \partial \B(2T_n)$, for
each $i\in \{1,2,\ldots, k(n)\}$. Note that $\sum_{i=1}^{k(n)}{\rm
Area}(\S(p_i(n),T_n))\leq {\rm Area} (\S\cap \B(3T_n)). $ Since $\S$
has no boundary, then ${\rm Area} (\S\cap \B(3T_n))\leq
\frac{4}{3}\pi c_2(3T_n)^3$ (see item~{\it \ref{n2}} of Theorem
\ref{T}). Therefore, we obtain that for all $n$, there exists an $i$,
such that
$$ {\rm Area }(\S(p_i(n),T_n))\leq \frac{c}{n}T_n^3,$$
for a fixed constant $c$. By definition of $A_{\inf}(\S,3)$, we
conclude that $A_{\inf}(\S,3)=0.$ Since we have shown that {\it
\ref{A3}} $\implies$ {\it \ref{sym}}, ${\cal T}(\S)$ contains a
minimal element $\S'$ with a plane of Alexandrov symmetry. Since
${\cal T}(\S)\subset {\cal T}(M)$, ${\cal T}(M)$ contains a minimal
element with a plane of Alexandrov symmetry. Thus ${\it  \ref{G3}
\implies \ref{sym}}$ which completes the proof of
 item~{\it \ref{inf3}}.

We next prove item~{\it \ref{infty}}. Assume that $M$ has an
infinite number of ends. If $M$ has empty boundary, then by the
arguments in the previous paragraph $A_{\inf}(M,3)=0$ and thus $\TM$
contains a minimal element with a plane of Alexandrov symmetry. By
Remark~\ref{rm25}, if $M$ has nonempty, compact boundary, then it
has a fixed size regular neighborhood on its mean convex side, which
is sufficient for item {\it 2} of Theorem \ref{T} to hold and then to apply the arguments in the previous paragraph. This
proves that item~{\it \ref{infty}} holds.

We next prove item~{\it \ref{bal_a}}. Arguing by contrapositive,
suppose that the conclusion of item~{\it \ref{bal_a}} fails to hold
and we will prove that $\TM$ contains an element with a plane of
Alexandrov symmetry. Since the conclusion of {\it \ref{bal_a}} fails
to hold, there exists a sequence of surfaces $\S(n)\in \Te(M)$ with
end representatives $E(n)$, and positive numbers $F(n)\to\infty$ as
$n\to\infty$ such that for any $R(n)>0$, there exist balls $B_n$  of
radius $F(n)$ such that
$$B_n\subset [\rth - (\B(R(n))\cup E(n))].$$

Choose $R(n)>F(n)$ sufficiently large so that $\partial E(n)\subset
\B(\frac{R(n)}{2})$. After rotating $B_n$ around an axis passing
through the origin, we obtain a new ball $K_n\subset \rth -
(\B(R(n))\cup E(n))$ of radius $F(n)$ such that $\partial K_n$
intersects $E(n)$ at a point $p_n$ of extrinsic distance at least $\frac{R(n)}{2}$ from $\partial E(n)$. After choosing a subsequence,
suppose that $E(n)-p_n$ converges to a surface $\S_{\infty}\in
\Te(M)$ which lies in a halfspace of $\rth$ which is a limit of some
subsequence of the translated balls $K_n-p_n$. By item~{\it
\ref{half}}, ${\cal T}(\S_\infty)\subset {\cal T}(M)$ contains a
surface with a plane of Alexandrov symmetry, which completes the
proof of item~{\it \ref{bal_a}}.

The proof of item~{\it \ref{bal_b}} is a modification of the proof
of item~{\it \ref{infty}}. In fact, if $\Sigma_n \in \Te(M)$ is a
sequence of surfaces with at least $n$ ends, $n$ going to infinity,
then $A_{\inf}(M,3)=0$, which implies that ${\cal T}(M)$ contains a
minimal element with a plane of Alexandrov symmetry.

 We now prove that item~{\it \ref{inf2}} holds. First observe
that ${\it \ref{Afinite1} \implies  \ref{A2}}$ and that ${\it
\ref{Gfinite1}\implies \ref{G2}}$. Also, equation \eqref{eq5}
implies that ${\it \ref{A2}\implies \ref{G2}}$ and that ${\it
\ref{Afinite1}\implies \ref{Gfinite1}}$. An argument similar to the
proof of ${\it \ref{sym} \implies \ref{Afinite2}}$ shows that ${\it
\ref{D} \implies \ref{Afinite1}}$. In order to complete the proof of
 item~{\it 9}, it suffices to show ${\it \ref{G2}\implies
\ref{D}}$.  Let $\S$ be a minimal element of $\TM$ satisfying~{\it
\ref{G2}}. By item~{\it \ref{inf3}}, there exists a minimal element
$\S'\in {\cal T}(\S)$ with a plane of Alexandrov symmetry. By
minimality of $\S$, $\S\in {\cal T}(\S')$, and so $\S$ also has a
plane $P$ of Alexandrov symmetry (the same plane as $\S'$ up to some
translation). In particular, $\S$ lies in a fixed slab in $\rth$.

After a possible rotation of $\S$, assume that $P=\{x_3=0\}$ and so,
$\Sigma \subset \{-a\leq x_3\leq a\}$ for some $a>0$. Since
$G_{\inf}(\S,2)=0$, there exists a sequence of points $p_n=(x_1(n),
x_2(n),0)\in \Sigma$, numbers $R_n$ with $R_n\to \infty $, such that
${\rm Genus} (\S(p_n,R_n))<\frac{1}{n}R_n^2.$ Similar to the proof
of ${\it \ref{G3} \implies \ref{sym}}$, the fact that $G_{\inf}(\S,
2)=0$ implies that we can find a sequence of points $q_n\in \B(p_n,
R_n)\cap P$ diverging in $\rth$ and numbers $k_n$, with $k_n\to
\infty$, such that one of the following statements holds.

\begin{enumerate}
\item $\mbox{Genus}(\S(q_n, k_n))=0.$
\item $\B(q_n, k_n)\subset [\S(p_n, \frac{R_n}{2})-\S(p_n,R_n)]$
and, as $n$ varies, there exist points $s_n\in
\partial \B(q_n,k_n)\cap \S(p_n, R_n)\cap P$ diverging in $\rth$.
\end{enumerate}

We will consider the two cases above separately. If statement  1
holds, then a subsequence of the translated surfaces $\S-q_n$ yields
a limit surface $\S_{\infty}\in {\cal T}(\S)$ of genus zero with $P$ as
a plane of Alexandrov symmetry. If $\S_{\infty}$ has a finite number
of ends, then it has an annular end. In this case, the end is
asymptotic to a Delaunay surface~\cite{kks1}. Therefore ${\cal
T}(\S)$ contains a Delaunay surface $\S'$ and since $\S$ is a
minimal element, $\S\in {\cal T}(\S')$ which implies $\S$ itself is
a Delaunay surface. Suppose $\S_\infty$ has an infinite number of
ends. Note that $\S_\infty$ lies in a slab which implies that ${\rm
Area}(\S_\infty\cap \B(R))\leq C_2R^2$ for some constant $C_2$. In
this case, a modification of the end of the proof that {\it
\ref{G3}} $\implies$ {\it \ref{sym}} shows that for each $n\in \N$,
there exist numbers $T_n$ with $T_n\to \infty$ such that the number
$k(n)$ of components $\{\S_1(T_n), \S_2(T_n),\ldots,
\S_{k(n)}(T_n)\}$ in $\S_{\infty} - \B(T_n)$ is at least $n$ and,
after possibly reindexing, there is a point $p_1(n)\in \S_1(T_n)\cap
\partial \B(2T_n)$, a constant $C$ such that ${\rm
Area}(\S_1(p_1(n), T_n)\leq \frac{C}{n} T_n^2$. This implies that
one can find diverging points $q_n\in \B(p_1(n), T_n)\cap P$ and
numbers $r_n\to \infty$ such that $\B(q_n,r_n)\subset [\B(p_1(n),
\frac{T_n}{2})-\S_1(p_1(n),T_n)]$ and there are points $s_n\in \partial{\B}(q_n, r_n)\cap \S_1(p_1(n), T_n)$ such that $|s_n|\to
\infty$. It follows that a subsequence of the surfaces $\S_1(p_1(n),
T_n)-s_n$ converges to a surface $\S_\infty$ which lies in halfspace
whose boundary plane is a vertical plane. Item {\it \ref{inf3}} of
Theorem \ref{T} implies that ${\cal T}(\S_\infty)$ contains a
surface $\S'$ with the plane $P$ as a plane of Alexandrov symmetry
as well as a vertical plane of Alexandrov symmetry. Therefore, $\S'$
is cylindrically bounded and so it is a Delaunay
surface~\cite{kks1}. Since $\S \in {\cal T}(\S')$, $\S$ is a
Delaunay surface~\cite{kks1}.

We now consider the case where statement 2 holds. A modification of
the proof of the case where statement 1 holds (this time translating
$\S$  by the points $-s_n$ instead by the points $-q_n$) then
demonstrates that there is a $\S'\in {\cal T}(\S)$ with both the plane
$P$ and a vertical plane as planes of Alexandrov symmetry. As
before, we conclude that $\S$ is a Delaunay surface. This completes
the proof of item~{\it 9}.

We now prove item~{\it \ref{oneend}}. Let $\S$ be a minimal element
in ${\cal T}(M)$. If $\S$ has a plane of Alexandrov symmetry and
$\Te(\S)$ contains a surface $\S'$ with more than one end, then
Theorem~{\it \ref{special}}, which does not depend on the proof of
this item, implies that $\S'$ has at least one annular end, from
which it follows that $\TS$ contains a Delaunay surface $D$. Since
$\S$ and $D$ are minimal elements of ${\cal T}(\S)$, then $\S\in
{\cal T}(\S)={\cal T}(D)$, and so $\S$ is a translation of $D$.
Since $\S$ is a Delaunay surface (a translation of $D$), then
clearly every surface in $\Te(\S)$ is also a translation of a
Delaunay surface, which proves item~{\it \ref{oneend}} under the
additional hypothesis that $\S$ has a plane of Alexandrov symmetry.

Thus, arguing by contradiction, suppose that $\S$ fails to have a
plane of Alexandrov symmetry and $\Te(\S)$ has a surface with more
than one end. Since $\S$ is a minimal element, then $\S\in {\cal
T}(\widetilde{\S})$ for any $\widetilde{\S}\in {\cal T}(\S)$, and so
no element of ${\cal T}(\S)$ has a plane of Alexandrov symmetry. By
 item~{\it \ref{bal_b}}, there is a bound on the number of ends of
any surface in $\Te(\S)$. Let $\S'\in \Te (\S)$ be a surface with
the largest possible number $n\geq2$ of ends and let $\{E_1, E_2,
\ldots, E_n\}$ be pairwise disjoint end representatives for its $n$
ends. By
 item~{\it \ref{bal_a}}, the ends $E_1, E_2,\ldots, E_n$ are
uniformly close to each other. It now follows from the definition of
$\Te(\S')$ that every element of ${\Te}(\S')$ must have at least $n$
components, each such component arising from a limit of translations of each of
the ends $E_1, E_2, \ldots, E_n$.

By our choice of $n$, we find that every surface in $\Te(\S')\subset
\Te (\S)$ has exactly $n$ components. From the minimality of $\S$,
$\S$ must be a component of some element $\S''\in \Te(\S')$. But
then our previous arguments imply $\Te(\S'')$ contains a surface
$\Delta$  with $n-1$ ends coming from translational limits of the
components of $\S''$ different from $\S$ and at least two additional
components (in fact $n$ components) arising from translational
limits of $\S \subset \S''$.  Hence, $\Te (\S'') \subset \Te (\S')$
contains a surface $\Delta$ with at least $n+1$ components, which
contradicts the definition of $n$.  This contradiction completes the
proof of item~{\it \ref{oneend}}.

We are now in a position to prove item~{\it \ref{sca}} of the
theorem. The first step in this proof is the following assertion.

\begin{assertion}\label{assca} Suppose $\S\in {\Te}(M)\cup\{M\}$
and every element in $\Te(\S)$ is connected. There exists a function
$f\colon [1,\infty)\to [1,\infty)$ such that for every
$\Omega\in\TS$ and for all points $p,q\in \Omega$ with $1\leq
d_{\Rsmall^3}(p,q)\leq R$, then
$$d_{\Omega}(p,q)\leq f(R) d_{\Rsmall^3}(p,q).$$ Furthermore, if
no element in $\Te(\S)$ has a plane of Alexandrov symmetry, then
$\S$ is chord-arc.
\end{assertion}

\begin{proof} Suppose $\S\in {\Te}(M)\cup\{M\}$ and every surface in $\Te(\S)$
is connected. If there fails  to exist the desired function $f$,
then there exists a positive number $R$, a sequence of surfaces
$\Omega(n)\in \Te(\S)$ and points $p_n,q_n\in \Omega (n)$ such that
for $n\in \N$,  $$1\leq d_{\Rsmall^3}(p_n, q_n)\leq R\;\;\;{\rm
and}\;\;\; n \cdot d_{\Rsmall^3} (p_n, q_n)\leq
d_{\Omega(n)}(p_n,q_n).
$$ Since by hypothesis every surface in $\Te(\S)$ is connected, $\Te(\S)={\cal
T}(\S)$. As ${\cal T}(\S)$ is sequentially compact and
$\TS=\Te(\S)$, the sequence of surfaces $\Omega(n)-p_n\in {\cal
T}(\S)$ can be chosen to converge to a $\S_\infty \in \Te(\S)=\TS$
and  the points $q_n-p_n$ converge to a point $q\in\S_\infty$.
Clearly $\S_\infty$ is not connected because it has a component
passing through $q$ and another component passing through the origin
(the intrinsic distance between $\vec{0}\in \Omega(n)-p_n$ and
$q_n-p_n\in \Omega(n)-p_n$ is at least $n$). But by assumption, every
surface in $\Te(\S)$ is connected. This contradiction proves the
existence of the desired function $f$.

Suppose now that $\TS$ contains no element with a plane of
Alexandrov symmetry and let $f$ be a function satisfying the first
statement in the assertion. Since $\S$ is an end representative of
$\S$ itself, item~{\it \ref{balls}} of the theorem implies that
there exists an $R_0>0$ such that every ball in $\rth$ of radius at
least $R_0$ intersects $\S$ in some point. Let $k$ be a positive
integer greater than $R_{0}+1$. Fix any two points $p,q\in \S$ of
extrinsic distance at least $4k$. Let $v=\frac{q-p}{|q-p|}$, $p_0=p$
and $p_{i+1}=p_i+2kv$, where $i\in\{0,1,\ldots, n-1\} $ and $q\in
\overline{\B}(p_n,k)$. By our choice of $k$, an open ball of radius $k-1$
always intersects $\S$ at some point. For each $0<i<n$, let $s_i\in
\S\cap \B (p_i,k-1)$; we choose $s_0=p$ and $s_n=q$. Since for
each $i<n$, $d_{\Rsmall^3}(s_i, s_{i+1})\leq 4k$ and $d_{\Rsmall^3}(s_i, s_{i+1})\geq 1$, then $d_{\S}(s_i,
s_{i+1})\leq f(4k)4k$. Using the triangle inequality and
$2(n-1)k\leq d_{\Rsmall^3}(p,q)$, we obtain
$$d_{\S}(p,q)\leq \sum_{i=0}^{n-1}d_{\S}(s_i, s_{i+1})\leq
nf(4k)4k\leq 2f(4k)( d_{\Rsmall^3}(p,q)+2k)< 4f(4k)
d_{\Rsmall^3}(p,q).$$ Thus, $\S$ is chord-arc with constant
$4f(4k)$, which completes the proof of the assertion.
\end{proof}

We now return to the proof of item~{\it \ref{sca}}. Let $\S \in \TM$
be a minimal element. By the last statement in item~{\it
\ref{oneend}}, the minimal element $\S$ satisfies $\TS=\Te(\S)$ and
so, every surface in $\Te(\S)$ is connected. Thus, by Assertion~{\it
\ref{assca}}, if $\S$ fails to have a plane $P$ of Alexandrov
symmetry, then $\S$ is chord-arc. Suppose now that $\S$ has a plane
$P$ of Alexandrov symmetry. If $\S$ were to fail to be chord-arc,
then the proof of item~{\it \ref{inf2}} shows that either $\S$ is a
Delaunay surface or else there exists an $R_0>0$ such that every
ball $B$ in $\rth$ of radius $R_0$ and centered at a point of $P$
must intersect $\S$. In the first case, $\S$ is a Delaunay surface,
which is clearly chord-arc. In the second case, the existence of
points in $B\cap \S$ allows one to modify the proof of Assertion
{\it \ref{assca}}  to show that $\S$ is chord-arc. Thus, item~{\it
\ref{sca}} of the theorem is proved.

In order to prove item~{\it \ref{n10}},  we need the following
lemma.

\begin{lemma}\label{lemma}
Let $\S$ be a minimal element in $\TM$. For all $D, \, \ve>0$, there
exists a $d_{\ve,D}>0$ such that the following statement holds. For
any $B_{\S}(p,D)\subset \S$ and for all $q\in \S$, there exists
$q'\in \Sigma$ such that $B_{\S}(q',D)\subset B_\Sigma(q,d_{\ve,D})$
and $d_{\cal H}(B_\Sigma(p,D)-p,\, B_\Sigma(q',D)-q')<\ve.$ Here
$B_\Sigma(p,R)$ denotes the intrinsic ball of radius $R$ centered at
$p$.
\end{lemma}

\begin{proof}
Arguing by contradiction, suppose that the claim in the lemma is
false. Then there exist $D, \, \ve>0$ such that the following holds.
For all $n\in \N$, there exist intrinsic balls
$B_\Sigma(p_n,D)\subset \Sigma$ and $q_n\in \Sigma$ such that for
any $B_\Sigma(q',D)\subset B_\Sigma(q_n, n)$, then $d_{\cal
H}(B_\Sigma(p_n,D)-p_n,B_\Sigma(q',D)-q')>\ve.$ In what follows, we
further simplify the notation and we let $B_\Sigma(p)$ denote
$B_\Sigma(p,D)$. After going to a subsequence, we can assume that
the set of translated surfaces, $\Sigma-p_n$, converges $C^2$ to a
complete, strongly Alexandrov embedded, $CMC$ surface
$\Sigma_\infty$ passing through the origin $\vec{0}$. By item~{\it
\ref{oneend}}, $\S_\infty$ is connected and we consider it to be
pointed so that $B_\Sigma(p_n)-p_n$ converges to
$B_{\Sigma_\infty}(\vec{0})$. Also, we can assume that
$B_\Sigma(q_n,n)-q_n$ converges to a complete, connected, pointed,
strongly Alexandrov embedded $CMC$ surface $\Sigma_\infty'$. The
previous discussion implies that for any $z\in \Sigma_\infty'$,
there exists a sequence $B_\Sigma(z_n)\subset B_\Sigma(q_n,n)$, such
that
\begin{equation}\label{eq1}
d_{\cal H}(B_\Sigma(z_n)-z_n,B_{\Sigma'_\infty}(z)-z)<\frac{\ve}{4}
\quad \text{for $n$ large}.
\end{equation}
Furthermore, we can also assume that
\begin{equation}\label{eq12}
d_{\cal
H}(B_\Sigma(p_n)-p_n,B_{\Sigma_\infty}(\vec{0}))<\frac{\ve}{4},
\end{equation} and since $B_\Sigma(z_n)\subset B_\Sigma(q_n,n)$,
then
\begin{equation}\label{eq2}d_{\cal H}(B_\Sigma(p_n)-p_n,B_\Sigma(z_n)-z_n)>\ve.\end{equation}

Recall that since $\Sigma$ is a minimal element, item~{\it \ref{n7}}
in Theorem \ref{T} implies that
$$\Sigma,\;\Sigma_\infty,\;\Sigma'_\infty\in {\cal T}(\Sigma)={\cal
T}(\Sigma_\infty)={\cal T}(\Sigma'_\infty).$$ In order to obtain a
contradiction it suffices to show that there exists an $\alpha>0$ such
that
$$d_{\cal
H}(B_{\Sigma'_\infty}(z)-z,B_{\Sigma_\infty}(\vec{0}))>\alpha$$ for
any $z\in\Sigma'_\infty$ because this inequality clearly implies
that $\Sigma_\infty \notin{\cal T}(\Sigma'_\infty)$. Fix
$z\in\Sigma'_\infty$ and let $z_n$ and $p_n$ be as given by
equations \eqref{eq1} and \eqref{eq12}.

In what follows, we are going to start with equation \eqref{eq2},
apply the triangle inequality for the Hausdorff distance between
compact sets, then apply the triangle inequality and equation
\eqref{eq1}, and finally we apply \eqref{eq12}. For $n$ large,
\begin{multline*}
\ve<d_{\cal H}(B_\Sigma(p_n)-p_n, B_\Sigma(z_n)-z_n)\leq \\
\leq d_{\cal H}(B_\Sigma(p_n)-p_n,B_{\Sigma'_\infty}(z)-z)+d_{\cal
H}(B_{\Sigma'_\infty}(z)-z,B_\Sigma(z_n)-z_n)<\\ <d_{\cal
H}(B_\Sigma(p_n)-p_n,B_{\Sigma_\infty}(\vec{0}))+d_{\cal
H}(B_{\Sigma_\infty}(\vec{0}),B_{\Sigma'_\infty}(z)-z)+\frac{\ve}{4}< \qquad\\
\qquad\qquad\qquad<\frac{\ve}{2}+d_{\cal
H}(B_{\Sigma'_\infty}(z)-z, B_{\Sigma_\infty}(\vec{0})).
\qquad\qquad\qquad\qquad\qquad\qquad
\end{multline*}

This inequality implies $d_{\cal
H}(B_{\Sigma'_\infty}(z)-z,B_{\Sigma_\infty}(\vec{0}))>\frac{\ve}{2}$,
which, after setting $\a=\frac{\ve}{2}$, completes the proof of the lemma.\end{proof}

Notice that if $X\subset \Sigma$ is a compact domain of intrinsic
diameter less than $D$, then for a point $p\in\Sigma$,  $X\subset
B_\Sigma(p,2D)$. The next lemma is a consequence of Lemma
\ref{lemma} and the following observation regarding the Hausdorff
distance: Given three compact sets $A,B,X\subset \Sigma$ with
$X\subset A$, if $d_{\cal H}(A,B)<\ve$, then there exists $X'\subset
B$ such that $d_{\cal H}(X,X')<\ve.$

\begin{lemma}
Let $\S$ be a minimal element of ${\cal T}(M)$. For all $D, \,
\ve>0$, there exists a $d_{\ve,D}>0$ such that the following
statement holds. For every smooth, connected compact domain $X\subset
\S$ with intrinsic diameter less than $D$ and for each $q\in \S$,
there exists a smooth compact, connected domain $X_{q,\ve}\subset
\S$ and a translation, $i\colon \rth \to \rth$, such that
$$d_{\S}(q,X_{q,\ve})<d_{\ve,D}\;\;\; \mbox{and} \;\;\; d_{\cal H}(X,
i(X_{q,\ve}))<\ve,$$ where $d_{\S}$ is distance function on $\S$.
\end{lemma}

In order to finish the proof of item~{\it \ref{n10}}, we remark that
 item~{\it \ref{sca}} implies intrinsic and extrinsic distances are
comparable when the intrinsic distance between the points is at
least one. Thus, the above lemma implies the first statement in
 item~{\it \ref{n10}}. The second statement is an immediate
consequence of the first statement, which completes the proof.

Theorem \ref{sp2} is now proved.
{\hfill\penalty10000\raisebox{-.09em}{$\Box$}\par\medskip}

\vspace{.2cm}

{\it Proof of Corollary \ref{cor2}.} We first prove item~{\it 1}\,
of the corollary. By equation \eqref{eq5}, $A_{\sup}(M,3)=0$ implies
$G_{\sup}(M,3)=0$. On the other hand, if $G_{\sup}(M,3)=0$, then for
any $\Sigma\in{\cal T}(M)$, $G_{\sup}(\Sigma,3)=0$. In particular,
for any minimal element $\S\in \TM$, $G_{\inf}(\Sigma,3)=0$. By
item~{\it \ref{inf3}}\, of Theorem \ref{sp2}, ${\cal T}(\Sigma)$
contains a minimal element $\Sigma'$ with a plane of Alexandrov
symmetry. Since $\Sigma$ is a minimal element, $\Sigma\in {\cal
T}(\Sigma')$ and therefore has a plane of Alexandrov symmetry. This
proves that item~{\it 1}\, holds.

The proof of item~{\it 2}\, follows from arguments similar to the
ones in the proof of item~{\it 1}, using item~{\it \ref{inf2}}\, of
Theorem \ref{sp2} instead of item~{\it \ref{inf3}}.
{\hfill\penalty10000\raisebox{-.09em}{$\Box$}\par\medskip}

\begin{remark}{\rm In regards to item {\it 4} of Theorem \ref{sp2}, it has been conjectured by Meeks \cite{me17} that if  $M$ is a properly embedded $CMC$ surface in $\rth$ which lies in the halfspace $\{x_3\geq 0\}$, then it has a horizontal plane of Alexandrov symmetry. This conjecture holds when $M$ has finite topology \cite{kks1} }\end{remark}

\begin{remark}\label{rm1}
{\rm In $\cite{mt5}$, we give a natural generalization of
Theorem~\ref{T} and~\ref{sp2} to the more general case of separating
$CMC$ hypersurfaces $M$ with bounded second fundamental form in an
$n$-dimensional noncompact homogeneous manifold $N$. In that paper,
we obtain some interesting applications of this generalization to
the classical setting where $N$ is $\mathbb{R}^n$ or hyperbolic
$n$-space, $\mathbb{H}^n$, which are similar to the applications
given in Theorem~\ref{sp2}.}
\end{remark}

\begin{remark}\label{rm3}
{\rm In~\cite{mt2}, we prove that if $M\subset \rth$ is a strongly
Alexandrov embedded $CMC$ surface with bounded second fundamental
form and $\TM$ contains a Delaunay surface, then $M$ is rigid. In
~\cite{smyt1}, Smyth and Tinaglia show that if $M$  contains a
surface with a plane of Alexandrov symmetry, then $M$ is locally
rigid\footnote{$M$ is {\it locally rigid} if any one-parameter
family of isometric immersions $M_t$ of $M$, $t\in [0,\ve)$,
$M_0=M$, with same mean curvature as $M$ is obtained by a family of
rigid motions of $M$.}. In relation to these rigidity results note
that Theorem \ref{sp2} gives several different constraints on the
geometry or the topology of $M$ that guarantee the existence of a
Delaunay surface or a surface with a plane of Alexandrov symmetry in
$\TM$. The first author conjectures that {\it the helicoid is the
only complete, embedded, constant mean curvature surface in $\rth$
which admits more than one noncongruent, isometric, constant mean
curvature immersion into $\rth$ with the same constant mean
curvature}. Since intrinsic isometries of the helicoid extend to
ambient isometries, this conjecture would imply that {\it an
intrinsic isometry of a complete, embedded, constant mean curvature
surface in $\rth$ extends to an ambient isometry of $\rth$}.}
\end{remark}

\section{Embedded $CMC$ surfaces with  a plane of Alexandrov
symmetry and more than one end.}
\label{sc4}

In this section we prove the following topological result that uses techniques from the proof of Theorem \ref{sp2}. In the next theorem the hypothesis that the surface $M$ be embedded can be replaced by the weaker condition that it is embedded in the complement of its Alexandrov plane of symmetry.

\begin{theorem}\label{special} Suppose $M$ is a not necessarily
connected, complete embedded $CMC$ surface with bounded second
fundamental form, possibly empty compact boundary, a plane of
Alexandrov symmetry, at least $n$ ends and every component of $M$ is
noncompact. If $n$ is at least two, then $M$ has at least $n$
annular ends. Furthermore, if $M$ has empty boundary and more than
one component, then each component of $M$ is a Delaunay surface.
\end{theorem}

The following corollary is an immediate consequence of the above
theorem and the result of Meeks~\cite{me17} that a connected,
noncompact, properly embedded $CMC$ surface with one end must have
infinite genus.

\begin{corollary} \label{corfinite} Suppose  $M$ is a connected,
noncompact, complete embedded $CMC$ surface with bounded second
fundamental form and a plane of Alexandrov symmetry. Then $M$ has
finite topology if and only if $M$ has a finite number of ends
greater than one.
\end{corollary}

In regards to Theorem \ref{special} when $n=\infty$, we note that
there exist connected surfaces of genus zero satisfying the
hypotheses of the theorem which are singly-periodic and have an
infinite number of annular ends. It is important to notice that the hypothesis in Theorem \ref{special} that $M$ has bounded second fundamental form is essential; otherwise, there are counterexamples (see Remark \ref{r4.7}).

\vspace{.2cm}

\begin{proof}
We first describe some of the notation that we will use in the proof
of the theorem.  We will assume that $M$ has a plane $P$ of
Alexandrov symmetry and $P$ is the $(x_1,x_2)$-plane.  We let
$\esf^1(R)=\partial (P\cap\B(R))$. Assume that $M$ is a bigraph over
a domain $\Delta \subset P$ and $R_0$ is chosen sufficiently large,
so that $\partial M \subset \B (R_0)$ and $\Delta -\B(R_0)$ contains
$n$ noncompact components $\Delta_1, \Delta_2, \ldots, \Delta_n$.
Let $M_1,\, M_2\subset M$ denote the bigraphs with boundary over the
respective regions $\Delta_1, \,\Delta_2$. Let $X$ be the component
in $P-(\De_1 \cup \De_2)$ with exactly two boundary curves
$\partial_1, \,\partial_2$, each a proper noncompact curve in $P$
and such that $\partial_1 \subset
\partial \De_1$, $\partial_2\subset \partial \De_2$. The curve
$\partial_1$ separates $P$ into two closed, noncompact,
simply-connected domains $P_1,\, P_2$, where $\De_1\subset P_1$ and
$\De_2\subset P_2$.

Now choose an increasing unbounded sequence of numbers
$\{R_n\}_{n\in \Nsmall}$ with $R_1>R_0$ chosen large enough so that
for $i=1,2,$ there exists a unique component of $P_i\cap \B (R_1)$
which intersects $P_i \cap \esf^1 (R_0)$ and so has $P_i \cap
\esf^1(R_0)$ in its boundary; we will also assume that the circles
$\esf^1 (R_n)$ are transverse to $\partial \De_1 \cup
\partial \De_2$ for each $n$. By elementary separation properties,
for $i=1,2$, there exists a unique component $\s_i(n)$ of $P_i \cap
\esf^1(R_n)$ which separates $P_i$ into two components, exactly one
of whose closure is a compact disk $P_i(n)$ with $P_i \cap
\esf^1(R_0)$ in its boundary; note that the collection of domains
$\{P_i(n)\}_n$ forms a compact exhaustion of $P_i$. See Figure
\ref{fig1}.

\begin{figure}[h!] \begin{center}
\includegraphics[width=3in]{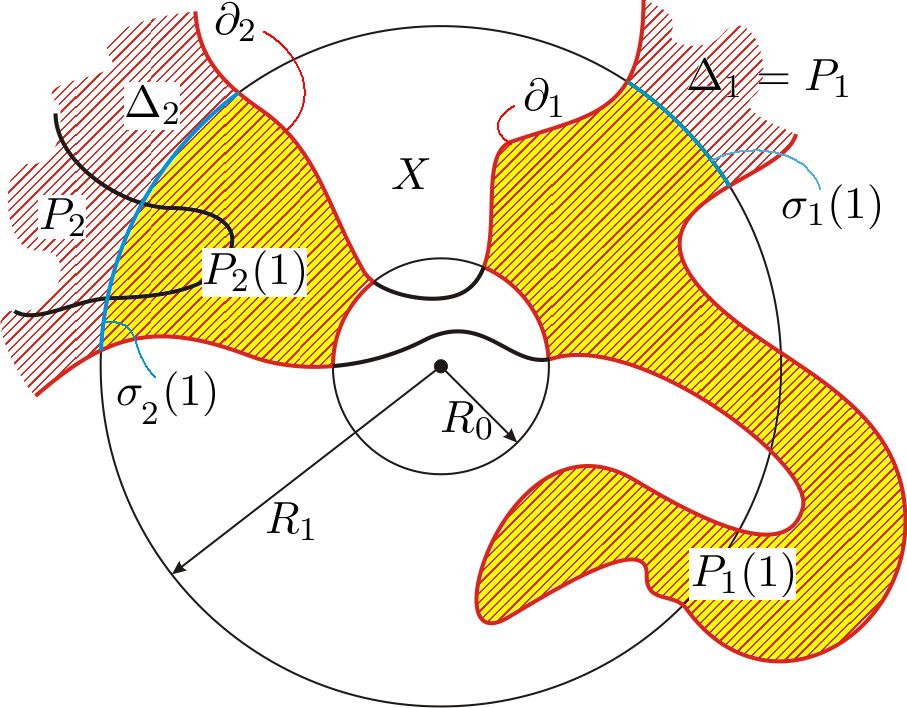} \end{center}
\caption{$P_1(1)$ is the yellow shaded region containing
$\sigma_1(1)$ and an arc of $\partial_1$ in its boundary. This
figure illustrates the possibility that $\Delta_1$ may equal $P_1$
while $\Delta_2$ may be strictly contained in $P_2$.} \label{fig1}
\end{figure}

Since $\s_1(n)$ is disjoint from $\s_2(n)$ and each of these sets is
a connected arc in $\esf^1(R_n)$, then, after possibly replacing the
sequence $\{ R_n\}_{n\in \Nsmall}$ by a subsequence and possibly
reindexing $P_1,P_2$, for each $n\in \N$, the arc $\s_1(n)$ is
contained in a closed halfspace $K_n$ of $\rth$ with
boundary plane $\partial K_n$ being a vertical plane passing through
the origin $\vec{0}$ of $\rth$. Let $\De_1(n)=\De_1\cap P_1 (n)$ and
let $\overline{M}_1(n)\subset M_1$ be the compact bigraph over $\De_1(n)$. Let
$\h{K}_n$ be the closed halfspace in $\rth$ with $K_n\subset
\h{K}_n$ and such that the boundary plane $\partial \h{K}_n$ is a
distance $\frac{2}{H}+ R_0$ from $\partial K_n$, where $H$ is the
mean curvature of $M$. Note that $\partial \overline{M}_1(n)$ is contained in
the union of the solid cylinder over $\overline{\B}(R_0)$ and the
halfspace $K_n$. Thus, the distance from $\partial \overline{M}(n)$ to
$\partial \h{K}_n$ is at least $\frac{2}{H}$. By the Alexandrov
reflection principle and the $\frac{1}{H}$ height estimate for $CMC$
graphs with zero boundary values and constant mean curvature $H$, we
find that $\overline{M}_1(n)\subset \h{K}_n$. After choosing a subsequence, the
halfspaces $\h{K}_n$ converge on compact sets of $\rth$ to a closed
halfspace $K$. Since for all $n\in\N$, $\overline{M}_1(n) \subset \overline{M}_1(n+1)$ and
$\bigcup_{n=1}^{\infty}\overline{M}_1(n) =M_1$, one finds that $M_1\subset K$.
After a translation in the $(x_1,x_2)$-plane and a rotation of $M_1$
around the $x_3$-axis, we may assume that the new surface, which we
will also denote by $M_1$, lies in $\{(x_1,x_2, x_3)\mid x_2>0 \}$
and it is a bigraph over a region $\De_1 \subset \{(x_1,x_2,0)\mid
x_2>0\}.$ A straightforward application of the Alexandrov reflection
principle and height estimates for $CMC$ graphs shows that, after an
additional translation in the $(x_1,x_2)$-plane and a rotation
around the $x_3$-axis, $\De_1$ also can be assumed to contain a
divergent sequence of points $p_n=(x_1(n), x_2(n), 0)\in
\partial \De_1$ such that $\frac{x_2(n)}{x_1(n)}\to 0$ as $n$
approaches infinity. See Figure \ref{fig2}.

\begin{figure}[h!] \begin{center}
\includegraphics[width=6in]{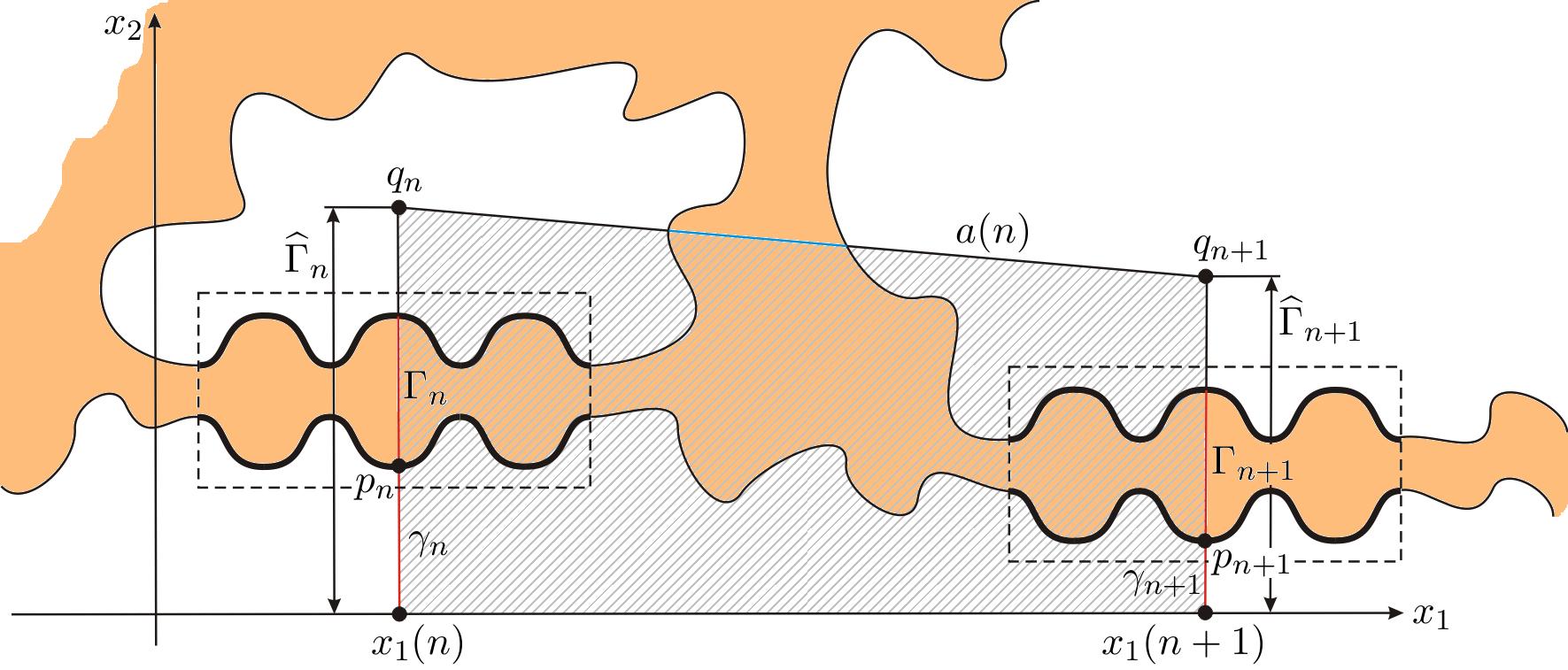} \end{center}
\caption{Choosing the points $p_n$ and related data. The
shaded trapezoidal region is $T(n)$.} \label{fig2}
\end{figure}

\begin{assertion}
\label{assertionD} The points $p_n$ can be chosen  to satisfy the
following additional properties:
\begin{enumerate}
\item The vertical line segments $\g_n$ joining $p_n$ to $(x_1(n),0,0)$
intersect $\De_1$ only at $p_n$ and $\frac{x_1(n+1)}{x_1(n)}>n$.
\item The surfaces $M_1-p_n$ converge to a surface in $\Te(M)$ with a
related component in ${\cal T}(M)$ being a Delaunay surface $F$ with $P$
as a plane of Alexandrov symmetry and axis parallel to the $x_1$-axis.

\end{enumerate}
\end{assertion}
\begin{proof} The proof that the points $p_n$ can be chosen to satisfy
statement 1 is clear. To prove that they can also be chosen to
satisfy statement 2 can be seen as follows. Let $S_n \subset P$ be
the circle passing through the points $p_n$ and
$(\frac{x_1(n)}{10},0,0)$ with center on the line $\{(x_1(n),s,0)
\mid s <x_2(n)\}$ and let $E_n$ denote the closed disk with boundary
$S_n$. Consider the family of translated disks $E_n(t)=E_n -(0,t,0)$
and let $t_0$ be the largest $t\geq 0$ such that $E_n(t)$ intersects
$\Delta_1$ at some point and let $D_n=E_n(t_0)$.  By construction
and after possibly replacing by a subsequence, points in $D_n \cap
\Delta_1$ satisfy the first statement in the assertion as well as
the previous property that the ratio of their $x_2$-coordinates to
their $x_1$-coordinates limit to zero as $n\to \infty$. Next replace
the previous point $p_n$ by any point of $\partial D_n \cap M_1$, to
obtain a new sequence of points which we also denote by $p_n$. A
subsequence of certain compact regions of the translated surfaces
$M-p_n$ converges to a strongly Alexandrov embedded surface $\Mint
\in \TM$ which has $P$ as a plane of Alexandrov symmetry and which
lies in the halfspace $x_2\geq 0$. It follows from item~{\it
\ref{half}} of Theorem \ref{sp2} (and its proof) that $\TS$ contains
a Delaunay surface $D$ with  axis being a bounded distance from the
$x_1$-axis and which arises from a limit of translates of
$M_\infty$. It is now clear how to choose the desired points
described in the assertion, which again we denote by $p_n$, so that
certain compact regions of the translated surfaces $M-p_n$ converge
to the desired Delaunay surface $F$. This completes the proof of the
assertion.
\end{proof}

As a reference for the discussion which follows, we refer the reader
to Figure \ref{fig2}. By Assertion \ref{assertionD}, we may assume
that around each point $p_n$, the surface $M_1$ is closely
approximated by a translation of a fixed large compact region of the
Delaunay surface $F$. Without loss of generality, we may assume that
the entire line containing $\g_n$ is disjoint from the
self-intersection set of $\partial \Delta_1$.  Let $\G_n$ be the
largest compact extension of $\g_n$ so that $\G_n-\g_n\subset \De_1$
and let $\h{\G}_n$ be a line segment extension of $\G_n$ near the
end point of $\G_n$ with positive $x_2$-coordinate so that $\h{\G}_n
\cap \De_1=\G_n \cap \De_1$ and so that the length of
$\h{\G}_n-\G_n$ is less than $\frac{1}{n}$. Let $q_n$ denote the end
point of $\widehat{\G}_n$ which is different from the point $p_n$.
Without loss of generality, we may assume that the line segments
$a(n)$ in $P$ joining  $q_n, q_{n+1}$ are transverse to $\partial
\De_1$ and intersect $\De_1$ in a finite collection of compact
intervals. If we denote by $v(n)$ the upward pointing unit vector in
the $(x_1,x_2)$-plane perpendicular to $a(n)$, then the vectors $v(n)$ converge
to $(0,1,0)$ as $n$ goes to infinity.

As a reference for the discussion which follows, we refer the reader
to Figure \ref{fig3}. Now fix some large $n$ and consider the
compact region $T(n)\subset P$ bounded by the line segments
$\h{\G}_n$, $\h{\G}_{n+1}$, $a(n)$ and  the line segment joining
$(x_1(n),0,0)$ to $(x_1(n+1),0,0)$. Consider $T(n)$ to lie in $\R^2$
and let $T(n)\times \R\subset \rth$ be the related convex domain in
$\rth$. Let $M_1(n)$ be the component of $M_1\cap (T(n)\times \R)$
which contains the point $p_n$. Note that $M_1(n)$ is compact with
boundary consisting of an almost circle $C(\G_n)$ which is a bigraph
over an arc in $\G_n$, possibly also an almost circle $C(\G_{n+1})$
which is a bigraph over an interval in $\G_{n+1}$ and a collection
of bigraph components over a collection of intervals $I_n$ in the
line segment $a(n)$.

We denote by $\a(n)$  the collection of boundary curves of $M_1(n)$.
Let $\a_2(n)$ be the subcollection of curves in $\a(n)$ which
intersect either $\G_n$ or $\G_{n+1}$, that is, $\a_2(n)=\{C(\G_n),
C(\G_{n+1})\}$ or $\a_2(n)=\{C(\G_n)\}$. Clearly, the collection of
boundary curves of $M_1(n)$ which are bigraphs over the collection
of intervals $I_n=\Delta_1 \cap a(n)$ is $\a(n)-\a_2(n)$. Let
$\a_3(n)$ be the subcollection of curves in $\a(n)-\a_2(n)$ which
bound a compact domain $\De (\a)\subset M_1-\partial M_1$, and let
$\a_4(n)=\a(n)-(\a_2(n)\cup \a_3(n))$. Note that in Figure \ref{fig3},
$\a_2(n)=\{C(\G_n), C(\G_{n+1})\}$, $\a_3(n)$ is empty and $\a_4(n)$
consists of the single blue curve $\partial$.

\begin{assertion}\label{limit} For $n$ sufficiently large,
every boundary curve $\partial$ of $M_1(n)$ which is a graph over an
interval in $I_n$, bounds a compact domain $\De(\partial)\subset
M_1-\partial M_1$; in other words, $\a_4(n)$ is empty.
\end{assertion}

\begin{proof}  For any $\a\in \a(n)$, let $\eta_\a$ denote
the outward pointing conormal to $\a\subset \partial M_1 (n)$ and
let $D(\alpha)$ be the planar disk bounded by $\a$. Consider a
boundary component $\partial\in \a_4(n)$. By the ``blowing a
bubble'' argument presented in ~\cite{kk2}, there exists another
disk $\h{D}(\partial)$ on the mean convex side of $M_1$ of the same
constant mean curvature as $M_1$, $\partial \h{D}(\partial)=\partial
D(\partial)$. Moreover, $\h{D}(\partial)$ is a graph over
$D(\partial)$ and $\h{D}(\partial)\cap (T(n)\times\R)=\partial
\h{D}(\partial)=\partial$. Let $\widehat{\eta}_{\partial}$ denote
the inward pointing conormal to $\partial\h{D}(\partial)$. The
graphical disk $\h{D}(\partial)$ is constructed so that $\langle
\eta_{\partial}-\widehat{\eta}_{\partial}, v(n)\rangle\geq0$, see
Figure \ref{fig3}.

\begin{figure}[h!]\begin{center}
\includegraphics[width=6in]{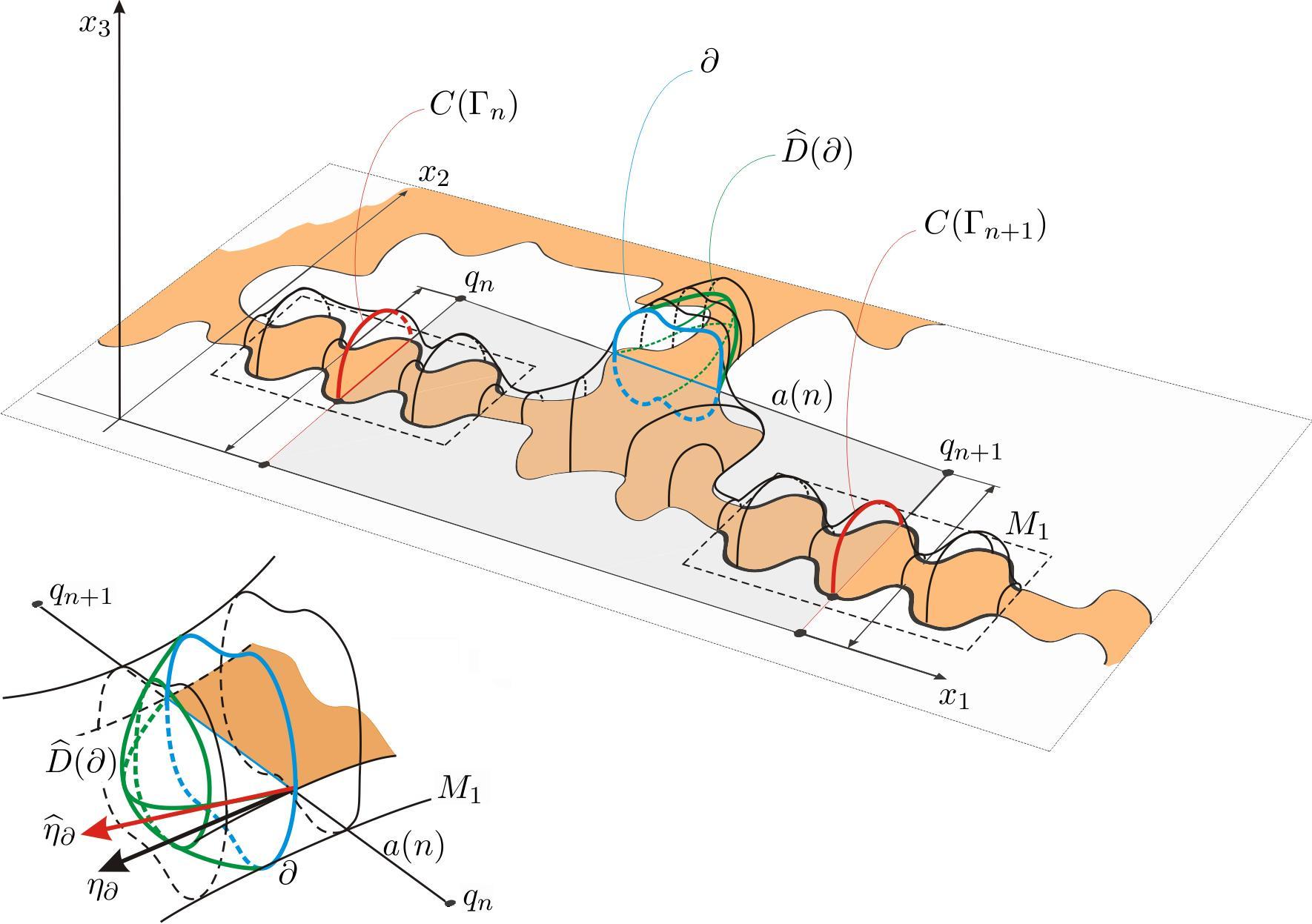} \end{center}
\caption{Blowing a bubble $\h{D}(\partial)$ on the mean convex side
of $M_1$.}\label{fig3}
\end{figure}

The piecewise smooth surface
$M_1(n)\cup(\bigcup_{\a\in\a_2(n)\cup\a_3(n)}D(\a))\cup
(\bigcup_{\a\in\a_4(n)}\h{D}(\a))$ is the boundary of a compact
region $W(n)\subset \rth$. An application of the
divergence theorem given in~\cite{kks1} to the vector field $v(n)$,
considered to be a constant vector field in $\rth$ in the region
$W(n)$, gives rise to the following equation:

\begin{multline} \label{eq7}
\sum_{\a\in \a_2(n)\cup\a_3(n)}\left[\int_{\a}\langle \eta_\a,
v(n)\rangle-2H\int_{D(\a)}\langle v(n), N(n)
\rangle\right]+\\+ \sum_{\partial\in \a_4(n)}\left[\int_{\partial}\langle
\eta_\partial, v(n)\rangle-2H\int_{\h{D}(\partial)}\langle v(n), N(n)
\rangle\right]=0,
\end{multline}
where $H$ is the mean curvature of $M$ and $N(n)$ is the outward
pointing conormal to $W(n)$.  Note that $\sum_{\a\in
\a_2(n)}\left[\int_{\a}\langle\eta_{\a}, v \rangle - 2H \int_{D(\a)}
\langle v(n), N \rangle \right] =\ve(n)$  converges to zero as $n\to
\infty$ because $v(n)$ converges to $(0,1,0)$ and the curves
$C(\G_n), C(\G_{n+1})$ converge to curves on Delaunay surfaces whose
axes are perpendicular to $(0,1,0)$. Also note that this application
of the divergence theorem in~\cite{kks1} implies that for $\a\in
\a_3(n)$, $\int_{\a}\langle \eta_{\a},
v(n)\rangle-2H\int_{D(\a)}\langle v(n), N(n)\rangle=0.$ Thus,
equation \eqref{eq7} reduces to the equation:
\begin{equation}\label{eq8}\ve(n)+\sum_{\partial\in \a_4(n)}
\left[\int_{\partial}\langle \eta_\partial,
v(n)\rangle-2H\int_{\h{D}(\partial)}\langle v(n), N(n)
\rangle\right]=0.\end{equation}

On the other hand, for each $\partial\in \a_4(n)$
\begin{equation}\label{eq9} \int_{\partial}\langle \eta_{\partial}, v(n)\rangle - 2H
\int_{\h{D}(\partial)}\langle v(n), N(n)\rangle
=\int_\partial\langle \eta_{\partial}-\widehat{\eta}_{\partial},
v(n)\rangle \geq 0\end{equation} and the length of each
$\partial\in\a_4$ is uniformly bounded from below. Since $\ve(n)$ is
going to zero as $n$ goes to infinity, equations \eqref{eq8} and
\eqref{eq9} above imply that for $n$ large, the conormals
$\eta_{\partial}$ and $\h{\eta}_{\partial}$ are approaching each
other uniformly (see the lower left hand corner of Figure \ref{fig3}). Note that the intrinsic
distance of any point on the graphs $\h{D}(\partial)$ to $\partial$
is uniformly bounded (independent of $\partial$ and
$n$)\footnote{This uniform intrinsic distance estimate holds since
$CMC$ graphs are strongly stable (existence of a positive Jacobi
function) and there are no strongly stable, complete $CMC$ surfaces
in $\rth$; see Theorem 2 in~\cite{ror1} for a proof of this
result.}. The Harnack inequality, the above remark, the facts that
$\h{D}(\partial)$ is simply-connected and the second fundamental
form of $M$ is bounded, imply that there exists $\de>0$ such that if
$\int_{\partial}\langle \eta_{\partial} -\h{\eta}_{\partial},
v(n)\rangle <\de$, then there is a disk $\De(\partial)\subset M_1 -
M_1(n)$ which can be expressed as a small graph over
$\h{D}(\partial)$. The existence of $\De(\partial)$ contradicts that
$\partial\in \a_4(n)$, which means $\a_4(n)=\mbox{\O}$ for $n$
sufficiently large. This contradiction proves the assertion.
\end{proof}

 We now apply Assertion \ref{limit} to prove the following key
partial result in the proof of Theorem \ref{special}.

\begin{assertion}  \label{final} $M_1$ has at least one annular end.
\end{assertion}

\begin{proof} By Assertion \ref{limit}, for some fixed $n$ chosen
sufficiently large, every boundary curve $\a$ of $M(n)$ in the
collection $\a(n)-\a_2(n)$ bounds a compact domain $\De(\a)\subset
M_1-\partial M_1$. By the Alexandrov reflection principle and height
estimates for $CMC$ graphs, we find that the surface
$\h{M}(n)=M(n)\cup \bigcup_{\a\in\a(n)-\a_2(n)}\Delta(\a)$ must have
two almost circles in its boundary arising from $\a_2(n)$. Let
$\S(k)=\bigcup_{j\leq k}\widehat{M}(n+j)$. Note that by the
Alexandrov reflection argument and height estimates for $CMC$ graphs
with zero boundary values, there exist half-cylinders $C(n,k)$ in
$\rth$ which contain $\S(k)$ and have fixed radii $\frac4H$. Hence
there is a limit half-cylinder $C(n)\subset \rth$ that contains
$\S(\infty)=\bigcup_{k\in \Nsmall}\S(k)\subset M$. By the main result in~\cite{kks1}, $\S(\infty)$ is asymptotic to
a Delaunay surface, which proves the assertion.
\end{proof}

It follows from the discussion at the beginning of the proof of
Theorem \ref{special} and Assertion \ref{final} that if $M$ has at
least $n$ ends, $n>1$, then it has at least $n-1$ annular ends. It
remains to prove that if $M_1, M_2$ are given as in the beginning of
the proof of Theorem \ref{special} with $M_1$ having an annular end,
then $M_2$ has an annular end as well. To see this note that the
annular end $E_1\subset M_1$ is asymptotic to the end $F$ of a Delaunay surface
and so after a rotation of $M$, $M_1$ is a graph over a domain
$\De_1$ which contains the axis of $F$, which we can assume
to be the positive $x_1$-axis. Now translate $M_2$ in the direction
$(-1,0,0)$ sufficiently far so that its compact boundary has
negative $x_1$-coordinates less than $-\frac{2}{H}$; call the
translated surface $M'_2$ and let $\De'_2\subset P$ be the domain
over which $M_2'$ is a bigraph. If for some $n\in \N$ the line
$L_n=\{(n,t,0)\mid t\in \R\}$ is disjoint from $M_2'$, then $M_2'$
is contained in a halfplane of $P$ and our previous arguments imply
$M_2'$ has an annular end. Thus without loss of generality, we may
assume that every line $L_n$ intersects  $\partial\De'_2$ a first
time at some point $s_n$ with positive $x_2$-coordinate.

For $\theta\in (0,\frac{\pi}{2}]$, let $r(\theta)$ be the ray with
base point the origin and angle $\theta$ and let $W(\theta)$ be the
closed convex wedge in $P$ bounded by $r(\theta)$ and the positive
$x_1$-axis. Let $\theta_0$ be the infimum of the set of $\theta \in
(0,\frac{\pi}{2}]$ such that $W(\theta)\cap\{s_n\}_{n\in \Nsmall}$
is an infinite set. Because of our previous placement of $\partial
M'_2$, a simple application of the Alexandrov reflection principle
and height estimates for $CMC$ graphs with zero boundary values
implies that some further translate $M_2''$ of $M_2'$ in the
direction $(-1,0,0)$ must be disjoint from $r(\theta_0)$. Finally,
after a clockwise rotation $\h{M}_2$ of $M''_2$ by angle $\theta_0$,
our previous arguments prove the existence of an annular end of
$\h{M}_2$ of bounded distance from the positive $x_1$-axis. Thus, we
conclude that $M_2$ also has an annular end which completes the
proof of the first statement in Theorem \ref{special}.

We next prove the last statement of the theorem. Suppose $M\subset
\rth$ is a complete, properly embedded $CMC$ surface with bounded
second fundamental form and with the $(x_1,x_2)$-plane $P$ as a
plane of Alexandrov symmetry. Suppose $M$ contains two noncompact
components $M_1, M_2$ and we will prove that each of these surfaces
is a Delaunay surface.

Consider $M_1$ and $M_2$ to be two disjoint end representatives of
$M$ defined as bigraphs over two disjoint connected domains
$\Delta_1, \Delta_2$ in $P$, respectively.
\begin{assertion}
After possibly reindexing $\Delta_1$, $\Delta_2$ and applying a
rigid of $\rth$ preserving the plane $P$, then $\Delta_1\subset \{x_2\geq
0\}$
\end{assertion}
\begin{proof}
By what we have proved so far, we know that $M_2$ has an annular end
$E$ which is asymptotic to the end $D(E)$ of a Delaunay surface. Let
$r_E\subset \Delta_2$ be a ray contained in the axis of $D(E)$.
After a rigid motion of $M$ preserving $P$ assume $r_E$ is a ray
based at the origin of $P$. The arguments used to prove the first
statement of the theorem show that there are two disjoint annular
ends $F,\,G$ of $M_1$ such that for $R$ large the arc
$\alpha(F,G,R)$ in $\esf^1(R)-\Delta_1$ which intersects $r_E$, has
one of its endpoints in $\Delta_1\cap F$ and its other endpoint in
$\Delta_1\cap G$. Let $D(F),\, D(G)$ be ends of Delaunay surfaces to
which $F,\, G$ are asymptotic. Let $r_F,\, r_G \subset \Delta_1$ be
rays contained in the axes of $D(F),\, D(G)$ respectively. Let
$\gamma_1$ be a properly embedded arc in $\Delta_1$ consisting of
$r_F,\, r_G$ and a compact arc joining their endpoints. Let
$\gamma_2'$ be the proper arc in $P- \B(R)$ consisting of
$\alpha(F,G,R)$, a boundary arc in $(E\cap \partial \Delta_1)
-\B(R)$ and a boundary arc in $(G\cap \partial \Delta_1)-\B(R)$.
After a small perturbation of $\gamma_2'$ we obtain a proper arc
$\gamma_2$ contained in $P-\Delta_1$ which intersects $r_E$. Note
that $\gamma_1$ is contained in a halfplane and since $\gamma_2$
lies at a bounded distance from $\gamma_1$, the halfplane can be
chosen to contain both $\gamma_1$ and $\gamma_2$. After a rigid
motion, we may assume that this halfplane is $\{x_2\geq 0\}$. Since
the region bounded by $\gamma_1$ and $\gamma_2$ is a strip by
construction, by elementary separation arguments, either $\gamma_1$
lies between $\{x_2=0\}$ and $\gamma_2$ or $\g_2$ lies between $\{x_2=0\}$ and $\g_1$. If $\gamma_1$
lies between $\{x_2=0\}$ and $\gamma_2$, then $\Delta_2$ lies in the
halfspace, otherwise $\Delta_1$ does. After possibly reindexing,
this completes the proof.
\end{proof}

In the discussion which follows, we refer the reader to
Figure~\ref{fig4}. By the previous assertion, we may assume
$\Delta_1\subset \{x_2\geq 0\}$. Previous arguments imply that after
a rigid motion of $M$, we can further assume that $M_1$ contains as
annular end $E^+$ with the property that for $n\in \N$ sufficiently
large, the line segments $\{(n,t,0)\mid t>0\}$ intersect $\Delta_1$
for a first time in a point $p_n\in E_1$. Furthermore, $E^+$ is
asymptotic to the end $D^+$ of a Delaunay surface. Also we can
assume that the half axis of revolution of $D^+$ lies in $P$ and is
a bounded distance from the positive $x_1$-axis.

By the Alexandrov reflection principle and height estimates for
$CMC$ graphs with zero boundary values, $\De_1$ must not be
contained in a convex wedge of $P$ with angle less than $\pi$.
Therefore, for $n\in \N$ sufficiently large the line segments
$\{(-n,t,0)\mid t>0\}$ intersect a second annular end $\De_1$ in
points $p_{-n}\in E^-$ for a first time.  In this case the annular
end $E^-$ is asymptotic to the end $D^-$ of another Delaunay surface
and the half axis of $D^-$ in $P$ is a bounded distance from the
negative $x_1$-axis.

\begin{figure}[h!]\begin{center}
\includegraphics[width=5.6in]{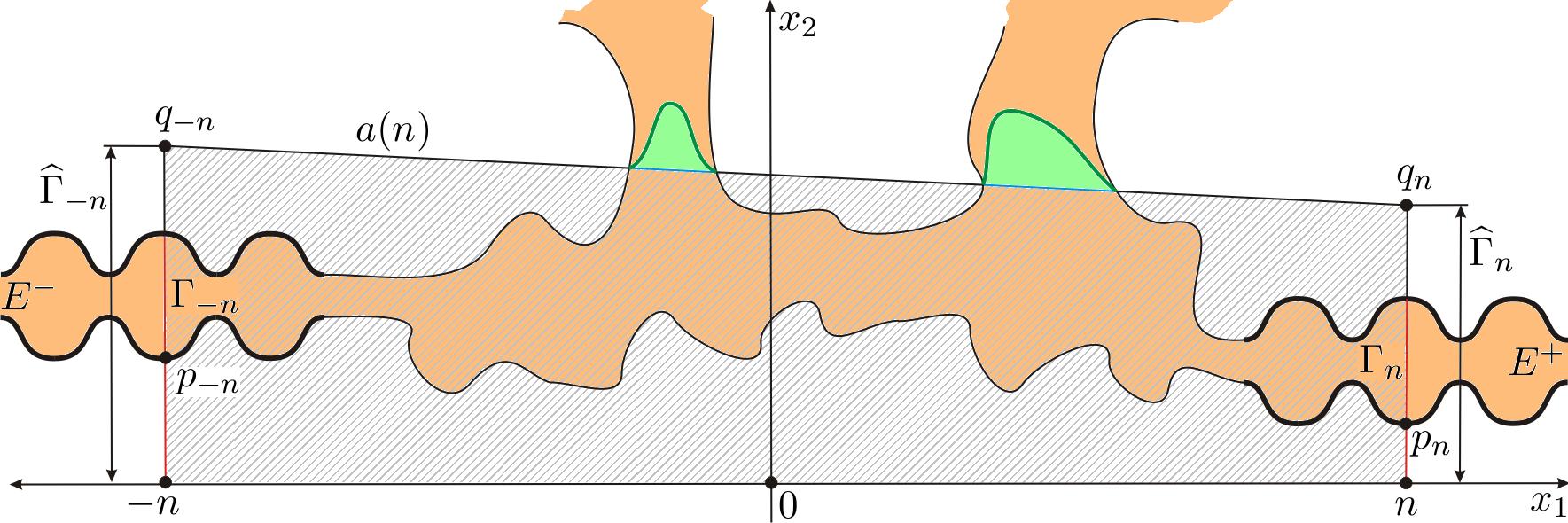}\end{center}
\caption{A picture of $M_1$ with two bubbles blown on its mean
convex side.}\label{fig4}
\end{figure}

Similar to our previous arguments, we define for each $n\in \Z$ with
$|n|$ sufficiently large, curves $\g_n, \G_n, \widehat{\G}_n$ and
points $q_n$ as we did before (see Figures \ref{fig2} and
\ref{fig4}). For each $n\in \N$ sufficiently large, we define the
line segment $a(n)\subset P$ whose end points are the points
$q_{-n},\, q_n$. Now define for any sufficiently large $n$, the
compact region $T(n)\subset P$ bounded by the line segments
$\widehat{\G}_{-n}$, $\widehat{\G}_n$, $a(n)$ and the line segment
joining $(-n, 0,0)$ to $(n,0,0)$ and let $T(n)\times \R\subset \rth$
be the related convex domain in $\rth$. Let $M_1(n)$ be the
component of $M_1\cap (T(n)\times \R)$ which contains the point
$p_{-n}$. Note that $M_1(n)$ is compact with boundary consisting of
an almost circle $C(\G_{-n})$ which is a bigraph over an arc on
$\G_{-n}$, and an almost circle $C(\G_n)$ which is a bigraph over an
arc on $\G_n$ and a collection of bigraph components over a
collection of intervals $I_n$ in the line segment $a(n)$.

As in previous arguments, an assertion similar to Assertion
\ref{limit} holds in the new setting. With this slightly modified
assertion, one finds that the almost circles $C(\G_{-n})$ and
$C(\G_n)$ bound a compact domain $\widehat{M}_1(n)\subset M_1$. A
slight modification of the proof of Assertion \ref{final} implies
$M_1$ is cylindrically bounded and so, by the main theorem in~\cite{kks1},
$M_1$ is a Delaunay surface. Note that the axis of $M_1$ is an
infinite line in $\De_1$ and so $\De_2$ also lies in a halfplane of
$P$. The arguments above prove that $M_2$ is also a Delaunay
surface, which completes the proof of the theorem.
\end{proof}

\begin{remark}{\rm \label{r4.7} Using techniques in~\cite{kap1,map},
for every integer $n>1$, it is
possible to construct a surface $M_n$ with empty boundary and $n$
ends, none of which are annular, which satisfies the hypotheses of
the surface $M$ in the statement of Theorem~\ref{special} except for
the bounded second fundamental form hypothesis. Hence, the
hypothesis in the theorem  that $M$ has bounded second fundamental
form is a necessary one in order for the conclusion of the theorem
to hold. }
\end{remark}

\center{William H. Meeks, III at bill@math.umass.edu}\\
Mathematics Department, University of Massachusetts, Amherst, MA
01003
\center{Giuseppe Tinaglia gtinagli@nd.edu   \\
Mathematics Department, University of Notre Dame, Notre Dame, IN,
46556-4618}

\nocite{kk2}
\bibliographystyle{plain}
\bibliography{bill}

\end{document}